\documentclass[10pt]{amsart}

\usepackage{amssymb,amsthm,amsmath,amsrefs, amsfonts,amscd}

\usepackage{dsfont}
\usepackage[all]{xy}

\usepackage{graphicx}

\numberwithin{equation}{section}

\theoremstyle{definition}
\newtheorem{theorem}{Theorem}[section]
\newtheorem{lemma}{Lemma}[section]
\newtheorem{proposition}{Proposition}[section]
\newtheorem{definition}{Definition}[section]
\newtheorem{corollary}{Corollary}[section]
\newtheorem{remark}{Remark}[section]
\newtheorem{example}{Example}[section]

\newcommand{\K}{\mathcal{K}}

\newcommand{\U}{\mathrm{U}}

\newcommand{\N}{\mathbb{N}}
\newcommand{\Z}{\mathbb{Z}}

\newcommand{\C}{\mathbb{C}}
\newcommand{\Cu}{\mathrm{Cu}}
\newcommand{\CCu}{\mathbf{Cu}}
\newcommand{\Hom}{\mathrm{Hom}}
\newcommand{\Aut}{\mathrm{Aut}}

\newcommand{\id}{\mathrm{id}}

\newcommand{\Ad}{\mathrm{Ad}}

\newcommand{\M}{\mathrm{M}}

\title[]{Equivariant *-homomorphisms, Rokhlin constraints and equivariant UHF-absorption}

\author{Eusebio Gardella}
\email{gardella@uoregon.edu}
\address{Department of Mathematics, Fenton Hall, University  of Oregon, Eugene OR 97403-1222, USA.}

\author{Luis Santiago}
\email{lmoreno@abdn.ac.uk}
\address{Institute of Mathematics, University of Aberdeen, Fraser Noble Building, Aberdeen AB24 3UE, UK}

\date{\today}

\begin{document}
\maketitle

\begin{abstract}
We classify equivariant *-homomorphisms between C*-dynamical systems associated to actions of finite groups with the Rokhlin property. In addition, the given actions are classified. An obstruction is obtained for the Cuntz semigroup of a C*-algebra allowing such an action. We also obtain an equivariant UHF-absorption result.
\end{abstract}

\tableofcontents

\section{Introduction}

Classification is a major subject in all areas of mathematics and has attracted the attention of many talented mathematicians. In the category
of C*-algebras, the program of classifying all amenable C*-algebras was initiated by Elliott, first with the classification of AF-algebras,
and later with the classification of certain simple C*-algebras of real rank zero. His work was followed by many other classification results for
nuclear C*-algebras, both in the stably finite and the purely infinite case.

The classification theory for von Neumann algebras precedes the classification program initiated by Elliott. In fact, the classification of
amenable von Neumann algebras with separable pre-dual, which is due to Connes, Haagerup, Krieger and Takesaki, was completed more than 30 years
ago. Connes moreover classified automorphisms of the type II$_1$ factor up to cocycle conjugacy in \cite{Connes-I}. This can be regarded as the
first classification result for actions on von Neumann algebras, and it was followed by his own work on the classification of pointwise outer
actions of amenable groups on von Neumann algebras in \cite{Connes-II}.

Several people have since then tried to obtain similar classification results for actions on C*-algebras. Early results in this direction
include the work of Herman and Ocneanu in \cite{Herman-Ocneanu} on integer actions with the Rokhlin property on UHF-algebras, the  work of Fack and
Mar\'echal in \cite{Fack-Marechal-I} and \cite{Fack-Marechal-II} for cyclic groups actions on UHF-algebras, and the work of Handelman and
Rossmann \cite{Handelman-Rossmann} for locally representable compact group actions on AF-algebras. Other results have been obtained by Elliott
and Su in \cite{Elliott-Su} for direct limit actions of $\Z_2$ on AF-algebras, and by Izumi in \cite{Izumi-I} and \cite{Izumi-II}, where he
proved a number of classification results for actions of finite groups on arbitrary unital separable C*-algebras with the Rokhlin property, as well as for approximately representable actions. The classification result of Izumi regarding actions with the Rokhlin property has been extended recently by Nawata in \cite{Nawata} to cover actions on certain not-necessarily unital separable C*-algebras (specifically for algebras $A$ such that $A\subseteq \overline{\mathrm{GL}(\widetilde{A})}$). It should
be emphasized that the classification of group actions on C*-algebras is a far less developed subject than the classification of C*-algebras
and even farther less developed than the classification of group actions on von Neumann algebras.

In this paper we extend the classification results of Izumi and Nawata of finite  group actions on C*-algebras with the Rokhlin property to actions of finite groups with the Rokhlin property on arbitrary separable C*-algebras. This is done by first obtaining a classification result for equivariant *-homomorphism between C*-dynamical systems associated to actions of finite groups with the Rokhlin property, and then applying Elliott's intertwining argument. In this paper we also obtain obstructions on the Cuntz semigroup, the Murray-von Neumann semigroup, and the K-groups of a C*-algebra allowing an action of a finite group with the Rokhlin property. These results are used together with the classification result of actions to obtain an equivariant UHF-absorption result.

This paper is organized as follows. In Section 2, we collect a number of definitions and results that will be used throughout the paper. In Section 3, we give an abstract classification for equivariant *-homomorphism between C*-dynamical systems associated to actions of finite groups with the Rokhlin property, as well as, a classification for the given actions. These abstract classification results are used together with known classification results of C*-algebras to obtain specific classification of equivariant *-homomorphisms and actions of finite groups on C*-algebras that can be written as inductive limits of 1-dimensional NCCW-complexes with trivial K$_1$-groups and for unital simple AH-algebras of no dimension growth.

In Section 4, we obtain obstructions on the Cuntz semigroup, the Murray-von Neumann semigroup, and the K$_*$-groups of a C*-algebra allowing an action of a finite group with the Rokhlin property. Then using the Cuntz semigroup obstruction we show that the Cuntz semigroup of a C*-algebra that admits an action of finite group with the Rokhlin property has certain divisibility property. In this section we also compute the Cuntz semigroup, the Murray-von Neumann semigroup, and the K$_*$-groups of the fixed-point and crossed product C*-algebras associated to an action of a finite group with the Rokhlin property.

In Section 5, we obtain divisibility results for the Cuntz semigroup of certain classes of C*-algebras and use this together with the classification results for actions obtained in Section 3 to prove an equivariant UHF-absorbing result.


\section{Preliminary definitions and results}
Let $A$ be a C*-algebra. We denote by $\M(A)$ its multiplier algebra, by $\widetilde{A}$ its unitization (that is, the C*-algebra obtained by adjoining a unit to $A$, even if $A$ is unital). If $A$ is unital, we denote by $\U(A)$ its unitary group. We denote by $\Aut(A)$ the automorphism group of $A$. The identity map of $A$ is denoted $\mathrm{id}_A$.

Topological groups are always assumed to be Hausdorff. If $G$ is a locally compact group and $A$ is a C*-algebra, then an \emph{action} of $G$ on $A$ is a strongly continuous group homomorphism $\alpha\colon G\to\Aut(A)$. Strong continuity for $\alpha$ means that for each $a$ in $A$, the map from $G$ to $A$ given by $g\mapsto\alpha_g(a)$ is continuous with respect to the norm topology on $A$.

We denote by $\K$ the C*-algebra of compact operators on a separable Hilbert space. We take $\N=\{1,2,\ldots\}$, $\mathbb{Z}_+=\{0,1,2,\ldots\}$, and $\overline{\mathbb{Z}_+}=\mathbb{Z}_+\cup \{\infty\}$.

\subsection{The Rokhlin property for finite group actions}\label{Rokhlin}
Let us briefly recall the definition of the Rokhlin property, in the sense of \cite[Definition 2]{SantiagoRP}, for actions of
finite groups on (not necessarily unital) C*-algebras. Actions with the Rokhlin property are the main object of study of this work.

\begin{definition}\label{def: Rokhlin}
Let $A$ be a C*-algebra and let $\alpha\colon G\to \Aut(A)$ be an action of a finite group $G$ on $A$. We say that $\alpha$ has
the \emph{Rokhlin property} if for any $\varepsilon>0$ and any finite subset $F\subseteq A$ there exist mutually orthogonal positive
contractions $r_g$ in $A$, for $g\in G$, such that
\begin{itemize}
\item[(i)] $\|\alpha_g(r_h)-r_{gh}\|<\varepsilon$ for all $g,h\in G$;
\item[(ii)] $\|r_ga-ar_g\|<\varepsilon$ for all $a\in F$ and all $g\in G$;
\item[(iii)] $\|(\sum_{g\in G} r_g)a-a\|<\varepsilon$ for all $a\in F$.
\end{itemize}
The elements $r_g$, for $g\in G$, will be called {\it Rokhlin elements} for $\alpha$ for the choices of $\varepsilon$ and $F$.
\end{definition}

It was shown in \cite[Corollary 1]{SantiagoRP} that Definition \ref{def: Rokhlin} agrees with \cite[Definition 3.1]{Izumi-I} whenever
the C*-algebra $A$ is unital. It is also shown in \cite[Corollary 2]{SantiagoRP} that Definition \ref{def: Rokhlin} agrees with
\cite[Definition 3.1]{Nawata} whenever the C*-algebra $A$ is separable.

If $A$ is a C*-algebra, we denote by $\ell^\infty(\N,A)$ the set of all bounded sequences $(a_n)_{n\in\N}$ in $A$ with the supremum norm $\|(a_n)_{n\in\N}\|=\sup\limits_{n\in\N}\|a_n\|$,
and pointwise operations. Then $\ell^\infty(\N,A)$ is a C*-algebra, and it is unital when $A$ is (the unit being the constant sequence $1_A$). Let
$$c_0(\N,A)=\left\{(a_n)_{n\in\N}\in\ell^\infty(\N,A)\colon \lim\limits_{n\to\infty}\|a_n\|=0\right\}.$$
Then $c_0(\N,A)$ is an ideal in $\ell^\infty(\N,A)$, and we denote the quotient
$$\ell^\infty(\N,A)/c_0(\N,A)$$
by $A^\infty$, which we call the \emph{sequence algebra} of $A$.

Write $\pi_A\colon \ell^\infty(\N,A)\to A^\infty$ for the quotient map, and identify $A$ with the subalgebra of $\ell^\infty(\N,A)$
consisting of the constant sequences, and with the subalgebra of $A^\infty$ by taking its image under $\pi_A$. We write $A_\infty=A^\infty\cap A'$
for the relative commutant of $A$ inside of $A^\infty$, and call it the \emph{central sequence algebra} of $A$.

Let $G$ be a finite group and let $\alpha\colon G\to\Aut(A)$ be an action of $G$ on $A$. Then there are actions of $G$ on $A^\infty$ and
$A_\infty$ which, for simplicity and ease of notation, and unless confusion is likely to arise, we denote simply by $\alpha$.

The following is a characterization of the Rokhlin property in terms of elements of the sequence algebra $A^\infty$ (\cite[Proposition 1]{SantiagoRP}):

\begin{lemma}\label{lem: Rokhlin equivalence}
Let $A$ be a C*-algebra and let $\alpha\colon G\to \Aut(A)$ be an action of a finite group $G$ on $A$. Then the following are equivalent:
\begin{itemize}
\item[(i)] $\alpha$ has the  Rokhlin property.
\item[(ii)] For any finite subset $F\subseteq A$  there exist mutually orthogonal positive contractions $r_g$ in $A^\infty\cap F'$, for $g\in G$, such that
\begin{itemize}
\item[(a)] $\alpha_g(r_h)=r_{gh}$ for all $g,h\in G$;
\item[(b)] $(\sum_{g\in G} r_g)b=b$ for all $b\in F$.
\end{itemize}
\item[(iii)] For any separable C*-subalgebra $B\subseteq A$ there are orthogonal positive contractions $r_g$ in $A^\infty\cap B'$ for $g\in G$
such that
\begin{itemize}
\item[(a)] $\alpha_g(r_h)=r_{gh}$ for all $g,h\in G$;
\item[(b)] $(\sum_{g\in G} r_g)b=b$ for all $b\in B$.
\end{itemize}
\end{itemize}
\end{lemma}

The first part of the following proposition is \cite[Theorem 2 (i)]{SantiagoRP}. The second part follows trivially from the definition of the
Rokhlin property.

\begin{proposition}\label{Rp properties}
Let $G$ be a finite group, let $A$ be a C*-algebra, and let $\alpha\colon G\to\Aut(A)$ be an action with the Rokhlin property.
\begin{enumerate}
\item If $B$ is any C*-algebra and $\beta\colon G\to\Aut(B)$ is any action of $G$ on $B$, then the action
$$\alpha\otimes\beta\colon G\to\Aut(A\otimes_{\mathrm{min}} B)$$
defined by $(\alpha\otimes\beta)_g=\alpha_g\otimes\beta_g$ for all $g\in G$, has the Rokhlin property.

\item If $B$ is a C*-algebra and $\varphi\colon A\to B$ is an isomorphism, then the action $g\mapsto \varphi\circ\alpha_g\circ\varphi^{-1}$ of $G$
on $B$ has the Rokhlin property.

\end{enumerate}
\end{proposition}

The following example may be regarded as the ``generating" Rokhlin action for a given finite group $G$. For some classes of C*-algebras, it can be shown that every action of $G$ with the Rokhlin property tensorially absorbs the action we construct below. See \cite[Theorems 3.4 and 3.5]{Izumi-II} and Theorem \ref{thm: Rp action absorbs model action} below.

\begin{example}\label{eg: model action} Let $G$ be a finite group. Let $\lambda\colon G\to  U(\ell^2(G))$ be the left regular representation, and identify $\ell^2(G)$ with $\C^{|G|}$. Define an action $\mu^G\colon G\to \Aut(\mathrm M_{|G|^\infty})$ by
$$\mu^G_g=\bigotimes_{n=1}^\infty \Ad(\lambda_g)$$
for all $g\in G$. It is easy to check that $\alpha$ has the Rokhlin property. Note that $\mu^G_g$ is approximately inner for all $g\in G$. \end{example}

It follows from part (i) of Proposition \ref{Rp properties} that any action of the form $\alpha\otimes\mu^G$ has the Rokhlin property.
One of our main results, Theorem \ref{thm: Rp action absorbs model action}, states that in some circumstances, every action with the Rokhlin
property has this form.

\subsection{The category $\CCu$, the Cuntz semigroup, and the $\Cu^\sim$-semigroup}

In this subsection, we will recall the definitions of the Cuntz and $\Cu^\sim$ semigroups, as well as the category $\CCu$, to which these semigroups naturally belong.

\subsubsection{The category $\CCu$}
Let $S$ be an ordered semigroup and let $s,t\in S$. We say that $s$ is \emph{compactly contained} in $t$, and denote this by $s\ll t$,
if whenever $(t_n)_{n\in\N}$ is an increasing sequence in $S$ such that $t\le \sup\limits_{n\in\N} t_n$, there exists $k\in \N$ such that $s\le t_k$.
 A sequence $(s_n)_{n\in\N}$ is said to be \emph{rapidly increasing} if $s_n\ll s_{n+1}$ for all $n\in \N$.

\begin{definition}\label{def: CCu} An ordered abelian semigroup $S$ is an object in the category $\CCu$ if it has a zero element and it satisfies the following properties:
\begin{itemize}
\item[(O1)] Every increasing sequence in $S$ has a supremum;
\item[(O2)] For every $s\in S$ there exists a rapidly increasing sequence $(s_n)_{n\in\N}$ in $S$ such that $s=\sup\limits_{n\in\N} s_n$.
\item[(O3)] If $(s_n)_{n\in\N}$ and $(t_n)_{n\in\N}$ are increasing sequences in $S$, then
$$\sup\limits_{n\in\N} s_n +\sup\limits_{n\in\N} t_n=\sup\limits_{n\in\N} (s_n+t_n);$$
\item[(O4)] If $s_1,s_2,t_1,t_2\in S$ are such that $s_1\ll t_1$ and $s_2\ll t_2$, then $s_1+s_2\ll t_1+t_2$.
\end{itemize}

Let $S$ and $T$ be semigroups in the category $\CCu$. An order-preserving semigroup map $\varphi\colon S\to T$ is a morphism in the category $\CCu$ if it preserves the zero element and it satisfies the following properties:
\begin{itemize}
\item[(M1)] If $(s_n)_{n\in\N}$ is an increasing sequence in $S$, then
$$\varphi\left(\sup\limits_{n\in\N} s_n\right)=\sup\limits_{n\in\N} \varphi(s_n);$$
\item[(M2)] If $s, t\in S$ are such that $s\ll t$, then $\varphi(s)\ll \varphi(t)$.
\end{itemize}
\end{definition}

It is shown in \cite[Theorem 2]{Coward-Elliott-Ivanescu} that the category $\CCu$ is closed under sequential inductive limits. The following description of inductive limits in the category $\CCu$ follows from the proof of this theorem.

\begin{proposition}\label{prop: inductivelimitCu}
Let $(S_n,\varphi_n)_{n\in\N}$, with $\varphi_n\colon S_n\to S_{n+1}$, be an inductive system in the category $\CCu$. For $m, n\in \N$
with $m\geq n$, let $\varphi_{n,m}\colon S_n\to S_{m+1}$ denote the composition $\varphi_{n,m}=\varphi_m\circ\cdots\circ\varphi_n$. A pair $(S,(\varphi_{n,\infty})_{n\in\N})$, consisting of a semigroup $S$ and morphisms $\varphi_{n,\infty}\colon S_n\to S$ in the category $\CCu$ satisfying $\varphi_{n+1,\infty}\circ\varphi_{n}=\varphi_{n,\infty}$ for all $n\in \N$, is the inductive limit of the system $(S_n,\varphi_n)_{n\in\N}$ if and only if:
\begin{itemize}
\item[(i)] For every $s\in S$ there exist elements $s_n\in S_n$ for $n\in\N$, such that $\varphi_n(s_n)\ll s_{n+1}$ for all $n\in \N$ and
$$s=\sup\limits_{n\in\N} \varphi_{n,\infty}(s_n);$$
\item[(ii)] Whenever $s,s',t\in S_n$ satisfy $\varphi_{n,\infty}(s)\le\varphi_{n,\infty}(t)$ and $s'\ll s$, there exists $m\ge n$ such that $\varphi_{n,m}(s')\le\varphi_{n,m}(t)$.
\end{itemize}
\end{proposition}

\begin{lemma}\label{lem: sup}
Let $S$ be a semigroup in $\CCu$, let $s$ be an element in $S$ and let $(s_n)_{n\in\N}$ be a rapidly increasing sequence in $S$ such that $s=\sup\limits_{n\in\N} s_n$. Let $T$ be a subset of $S$ such that every element of $T$ is the supremum of a rapidly increasing sequence of elements in $T$. Suppose that for every $n\in \N$ there is $t\in T$ such that $s_n\ll t\le s$. Then there exists an increasing sequence $(t_n)_{n\in\N}$ in $T$ such that $s=\sup\limits_{n\in\N} t_n$.
\end{lemma}
\begin{proof}
It is sufficient to construct an increasing sequence $(n_k)_{k\in\N}$ of natural numbers and a sequence $(t_k)_{k\in\N}$ in $T$ such
that $s_{n_k}\le t_k\le s_{n_{k+1}}$ for all $k\in \N$, since this implies that $s=\sup\limits_{k\in\N} t_k$.\\
\indent For $k=1$, set $n_1=1$ and $s_{n_1}=0$. Assume inductively that we have constructed $n_j$ and $t_j$ for all $j\le k$ and let us
construct $n_{k+1}$ and $t_{k+1}$. By the assumptions of the lemma, there exists $t\in T$ such that $s_{n_k}\ll t\le s$. Also by assumption,
$t$ is the supremum of a rapidly increasing sequence of elements of $T$. Hence there exists $t'\in T$ such that $s_{n_k}\le t'\ll s$. Use
that $s=\sup\limits_{n\in\N} s_n$ and $t'\ll s$, to choose $n_{k+1}\in\N$ with $n_{k+1}>n_k$ such that $t'\le s_{n_{k+1}}\ll s$. Set $t_{k+1}=t'$.
Then $s_{n_k}\le t_{k+1}\le s_{n_{k+1}}$. This completes the proof of the lemma.
\end{proof}

\begin{definition} \label{df: S gamma}
Let $S$ be a semigroup in the category $\CCu$. Let $I$ be a nonempty set and let $\gamma_i\colon S\to S$ for $i\in I$, be a family of
endomorphisms of $S$ in the category $\CCu$. We introduce the following notation:
$$S^\gamma=\left\{s\in S\colon \exists \ (s_t)_{t\in(0,1]} \mbox{ in } S \colon
\begin{aligned}
& s_r\ll s_{t} \mbox{ if } r<t,\, s_t=\sup\limits_{r<t} s_r \ \forall \ t\in (0,1],\\
& s_1=s, \mbox{ and } \gamma_i(s_t)=s_t  \ \forall \ t\in (0, 1] \mbox{ and } \forall\  i\in I
\end{aligned}
\right\},$$
and
$$S^\gamma_\N=\left\{s\in S\colon \exists \ (s_n)_{n\in\N} \mbox{ in } S\colon
\begin{aligned}
& s_n\ll s_{n+1} \ \forall \ n\in\N, \, s=\sup\limits_{n\in\N} s_n,\\
& \mbox{ and } \gamma_i(s_n)=s_n \ \forall \ n\in\N \mbox{ and } \forall\  i\in I
\end{aligned}
\right\}.$$
\end{definition}

\begin{lemma}\label{lem: closure}
Let $S$ be a semigroup in the category $\CCu$. Let $I$ be a nonempty set and let $\gamma_i\colon S\to S$ for $i\in I$, be a family of
endomorphisms of $S$ in the category $\CCu$. Then
\begin{itemize}
\item[(i)] $S^\gamma_\N$ is closed under suprema of increasing sequences;
\item[(ii)] $S^\gamma$ is an object in $\CCu$.
\end{itemize}
\end{lemma}
\begin{proof} (i). Let $(s_n)_{n\in\N}$ be an increasing sequence in $S_{\N}^\gamma$. For each $n\in \N$, choose a rapidly increasing
sequence $(s_{n,m})_{m\in\N}$ in $S$ such that $s_n=\sup\limits_{m\in\N} s_{n,m}$ and $\gamma_i(s_{n,m})=s_{n,m}$ for all $i\in I$
and $m\in \N$. By the definition of the compact containment relation, there exist increasing sequences $(n_j)_{j\in\N}$ and $(m_j)_{j\in\N}$
in $\N$ such that $s_{k,l}\le s_{n_j,m_j}$ whenever $1\le k,l\le j$, and such that $( s_{n_j,m_j})_{j\in \N}$ is increasing. Let $s$ be the supremum of $(s_{n_j,m_j})_{j\in\N}$ in $S$. Then
$s\in S_\N^\gamma$, and it is straightforward to check, using a diagonal argument, that $s=\sup\limits_{n\in\N} s_n$, as desired.

(ii). It is clear that $S^\gamma$ satisfies O2, O3 and O4.
Now let us check that $S^\gamma$ satisfies axiom O1. Let $(s^{(n)})_{n\in\N}$ be an increasing sequence in $S^\gamma$ and let $s$ be its supremum in $S$. It is sufficient to show that $s\in S^\gamma$.

For each $n \in \N$, choose a path $(s_t^{(n)})_{t\in (0,1]}$ as in the definition of $S^\gamma$ for $s^{(n)}$.
Using that $s_t^{(n)}\ll s^{(n+1)}$ for all $n\in\N$ and all $t\in (0,1)$, together with a diagonal argument,
choose an increasing sequence $(t_n)_{n\in\N}$ in $(0,1]$ converging to $1$, such that
$$s_{t_n}^{(n)}\ll s_{t_{n+1}}^{(n+1)} \quad\forall \ n\in \N, \mbox{ and } s=\sup\limits_{n\in\N}s_{t_n}^{(n)}.$$
This implies, using the definition of the compact containment relation, that for each $n\in \N$ there exists $t_{n+1}'$ such that
$t_n<t_{n+1}'<t_{n+1}$ and
$$s_{t_n}^{(n)}\ll s_t^{(n+1)}\le s_{t_{n+1}}^{(n+1)} \mbox{ for all } t\in (t_{n+1}', t_{n+1}].$$
Choose an increasing function $f\colon (0,1]\to (0,1]$ such that
$$f\left(\left(1-\frac{1}{n}, 1-\frac{1}{n+1}\right]\right)=(t'_{n+1}, t_{n+1}]$$
for all $n\in \N$.
Define a path $(s_t)_{t\in (0,1]}$ in $S$ by taking $s_1=s$ and
$$s_t=s_{f(t)}^{(n+1)} \mbox{ for } t\in \left(1-\frac{1}{n}, 1-\frac{1}{n+1}\right].$$
Then $\gamma_i(s_t)=s_t$ for all $t\in (0,1]$ and all $i\in I$, so $s\in S^\gamma$. It is clear that this path satisfies the conditions in the definition of $S^\gamma$ for $s$.
\end{proof}

\subsubsection{The Cuntz semigroup}

Let $A$ be a C*-algebra and let $a, b\in A$ be
positive elements. We say that $a$ is \emph{Cuntz subequivalent} to $b$, and denote this by $a\precsim b$, if there exists a sequence $(d_n)_{n\in\N}$
in $A$ such that $\lim\limits_{n\to \infty} \|d_n^*bd_n- a\|=0$. We say that $a$ is \emph{Cuntz equivalent} to $b$, and denote this by $a\sim b$,
if $a\precsim b$ and $b\precsim a$. It is clear that $\precsim$ is a preorder relation in the set of positive elements of $A$, and thus
$\sim$ is an equivalence relation. We denote by $[a]$ the Cuntz equivalence class of the element $a\in A_+$.

The first conclusion of the following lemma was proved in \cite[Proposition 2.2]{Rordam} (see also \cite[Lemma 2.2]{Kirchberg-Rordam}).
The second statement was shown in \cite[Lemma 1]{Robert-Santiago}.

\begin{lemma}\label{lem: Cuntz relation}
Let $A$ be a C*-algebra and let $a$ and $b$ be positive elements in $A$ such that $\|a-b\|<\varepsilon$. Then $(a-\varepsilon)_+\precsim b$.
More generally, if $r$ is a non-negative real number, then $(a-r-\varepsilon)_+\precsim (b-r)_+$.
\end{lemma}

The Cuntz semigroup of $A$, denoted by $\Cu(A)$, is defined as the set of Cuntz equivalence classes of positive elements of $A\otimes \K$.
Addition in $\Cu(A)$ is given by
$$[a]+[b]=[a'+b'],$$
where $a',b'\in (A\otimes\K)_+$ are orthogonal and satisfy $a'\sim a$ and $b'\sim b$. Furthermore, $\Cu(A)$ becomes an ordered semigroup when
equipped with the order $[a]\le [b]$ if $a\precsim b$. If $\phi\colon A\to B$ is a *-homomorphism, then $\phi$ induces an order-preserving
map $\Cu(\phi)\colon \Cu(A)\to \Cu(B)$, given by $\Cu(\phi)([a])=[(\phi\otimes \mathrm{id}_\K)(a)]$ for every $a\in (A\otimes \K)_+$.

\begin{remark}
Let $A$ be a C*-algebra, let $a\in A$ and let $\varepsilon>0$. It can be checked that $[(a-\varepsilon)_+]\ll [a]$ and that
$[a]=\sup\limits_{\varepsilon>0} [(a-\varepsilon)_+]$, thus showing that $\Cu(A)$ satisfies Axiom O2.
\end{remark}

It is shown in \cite[Theorem 1]{Coward-Elliott-Ivanescu} that $\Cu$ is a functor from the category of C*-algebras to the category $\CCu$.

\begin{lemma}\label{lem: morphism}
Let $A$ and $B$ be C*-algebras and let $\rho\colon \Cu(A)\to \Cu(B)$ be an order-preserving semigroup map. Suppose that for all $a\in (A\otimes\K)_+$ one has
\begin{itemize}
\item[(i)] $\rho([a])=\sup\limits_{\varepsilon>0}\rho([(a-\varepsilon)_+])$,
\item[(ii)] $\rho([(a-\varepsilon)_+])\ll \rho ([a])$ for all $\varepsilon>0$.
\end{itemize}
Then $\rho$ is a morphism in the category $\CCu$; that is, it preserves suprema of increasing sequences and the compact containment relation.
\end{lemma}
\begin{proof}
We show first that $\rho$ preserves suprema of increasing sequences. Let $a$ be a positive element in $A\otimes\K$ and
let $(a_n)_{n\in\N}$ be an increasing sequence of positive elements in $A\otimes\K$ such that $\sup\limits_{n\in\N}[a_n]=[a]$.
Then $\rho([a_n])\le \rho([a])$ for all $n\in\N$. Suppose that $b\in (B\otimes\K)_+$ is such that $\rho([a_n])\le [b]$ for
all $n\in \N$ and let $\varepsilon>0$. By the definition of the compact containment relation and the fact that $[(a-\varepsilon)_+]\ll [a]$,
there exists $n_0\in \N$ such that $[(a-\varepsilon)_+]\le [a_{n_0}]$. By applying $\rho$ to this inequality we get
$$\rho([(a-\varepsilon)_+])\le \rho([a_{n_0}])\le [b].$$
By taking supremum in $\varepsilon>0$ and applying (i) we get
$$\rho([a])=\sup\limits_{\varepsilon>0}\rho([(a-\varepsilon)_+])\le [b].$$
This shows that $\rho([a])$ is the supremum of $(\rho([a_n]))_{n\in \N}$, as desired.

We proceed to show that $\rho$ preserves the compact containment relation. Let $a$ and $b$ be positive elements in $A\otimes\K$ such
that $[a]\ll [b]$. Choose $\varepsilon>0$ such that $[a]\le [(b-\varepsilon)_+]\le [b]$. It follows that
$$\rho([a])\le \rho([(b-\varepsilon)_+])\le \rho([b]).$$
By (ii) applied to $[b]$ we get $\rho([a])\ll \rho([b])$, which concludes the proof.
\end{proof}

The following lemma is a restatement of \cite[Lemma 4]{Robert-Santiago}.

\begin{lemma}\label{lem: interpolation}
Let $A$ be a C*-algebra, let $(x_i)_{i=0}^n$ be elements of $\Cu(A)$ such that $x_{i+1}\ll x_i$ for all $i=0,\ldots,n$, and let $\varepsilon>0$.
Then there exists $a\in (A\otimes \K)_+$ such that
\begin{align*}
x_n\ll [(a-(n-1)\varepsilon)_+]&\ll x_{n-1}\ll  [(a-(n-2)\varepsilon)_+]\ll\cdots\\
\cdots\ll &x_3\ll [(a-2\varepsilon)_+]\ll x_2\ll [(a-\varepsilon)_+] \ll x_1\ll [a]=x_0.
\end{align*}
\end{lemma}

\subsubsection{The $\Cu^\sim$-semigroup}
Here we define the $\Cu^\sim$-semigroup of a C*-algebra. This semigroup was introduced in \cite{Robert} in order to classify certain inductive limits of 1-dimensional NCCW-complexes.

\begin{definition}\label{df: Cu sim}
Let $A$ be C*-algebra and let
$\pi\colon \widetilde{A} \to \widetilde{A}/A\cong\C$ denote the quotient map.
Then $\pi$ induces a semigroup homomorphism
$$\mathrm{Cu}(\pi)\colon \mathrm{Cu}(\widetilde{A})\to \mathrm{Cu}(\C)\cong \overline{\mathbb{Z}_+}.$$
We define the semigroup $\Cu^\sim(A)$ by
\[
\mathrm{Cu}^\sim(A)=\{([a], n) \in \Cu(\widetilde{A})\times \mathbb{Z}_+ \mid \mathrm{Cu}(\pi)([a])=n\}/\sim,
\]
where $\sim$ is the equivalence relation defined by
\[
([a], n)\sim ([b],m) \quad\text{ if }\quad  [a]+m[1]+k[1]= [b]+n[1]+k[1],
\]
for some $k\in \N$. The
image of the element $([a], n)$ under the canonical quotient map is denoted by $[a]-n[1]$.

Addition in $\mathrm{Cu}^\sim(A)$ is induced by pointwise addition in $\Cu(\widetilde{A})\times\mathbb{Z}_+$.
The semigroup $\Cu^\sim(A)$ can be endowed with an order: we say that $[a]-n[1]\le [b]-m[1]$ in $\Cu^\sim(A)$ if there exists $k$ in $\mathbb{Z}_+$ such that
$$[a]+(m+k)[1]\le [b]+(n+k)[1]$$
in $\Cu(\widetilde{A})$.

The assignment $A\mapsto \Cu^\sim(A)$ can be turned into a functor as follows.
Let $\phi\colon A\to B$ be a *-homomorphism and let $\widetilde{\phi}\colon \widetilde{A}\to \widetilde{B}$ denote the unital extension of
$\phi$ to the unitizations of $A$ and $B$. Let us denote by $\Cu^\sim(\phi)\colon \Cu^\sim(A)\to \Cu^\sim(B)$ the map defined by
$$\Cu^\sim(\phi)([a]-n[1])=\Cu(\widetilde{\phi})([a])-n[1].$$
It is clear that $\Cu^\sim(\phi)$ is order-preserving, and thus $\Cu^\sim$ becomes a functor from the category of C*-algebras to
the category of ordered semigroups.
\end{definition}

It was shown in \cite{Robert} that the $\Cu^\sim$-semigroup of a C*-algebra with stable rank one
belongs to the category $\CCu$, that $\Cu^\sim$ is a functor from the category of C*-algebras of
stable rank one to the category $\CCu$, and that it preserves inductive limits of sequences.

\section{Classification of actions and equivariant *-homomorphisms}

In this section we classify equivariant *-homomorphisms whose codomain C*-dynamical system have the Rokhlin property. We use this results to classify actions of finite groups on separable C*-algebras with the Rokhlin property.
Our results complement and extend those obtained by Izumi in \cite{Izumi-I} and \cite{Izumi-II} in the unital setting, and by Nawata in \cite{Nawata} for C*-algebras $A$ that satisfy $A\subseteq \overline{\mathrm{GL}(\widetilde{A})}$.

\subsection{Equivariant *-homomorphisms}

Let $A$ and $B$ be C*-algebras and let $G$ be a compact group. Let $\alpha\colon G\to \Aut(A)$ and $\beta\colon G\to \Aut(B)$ be (strongly continuous) actions. Recall that a *-homomorphism $\phi\colon A\to B$ is said to be \emph{equivariant} if $\phi\circ\alpha_g=\beta_g\circ\phi$ for all $g\in G$.

\begin{definition}\label{df: B^G}
Let $A$ and $B$ be C*-algebras and let $\alpha\colon G\to \Aut(A)$ and $\beta\colon G\to \Aut(B)$ be actions of a compact group $G$. Let $\phi, \psi\colon A\to B$ be equivariant *-homomorphisms. We say that $\phi$ and $\psi$ are \emph{equivariantly approximately unitarily equivalent}, and denote this by $\phi\sim_{\mathrm{G-au}}\psi$, if for any finite subset $F\subseteq A$ and for any $\varepsilon>0$ there exists a unitary $u\in \widetilde{B^\beta}$ such that
$$\|\phi(a)-u^*\psi(a)u\|<\varepsilon,$$
for all $a\in F$.
\end{definition}

Note that when $G$ is the trivial group, this definition agrees with the standard definition of approximate unitary equivalence of *-homomorphisms. In this case we will omit the group $G$ in the notation $\sim_{\mathrm{G-au}}$, and write simply $\sim_{\mathrm{au}}$.

The following lemma can be proved using a standard semiprojectivity argument. Its proof is left to the reader.

\begin{lemma}\label{lem: unitaries can be lifted}
Let $A$ be a unital C*-algebra and let $u$ be a unitary in $A_\infty$. Given $\varepsilon>0$ and given a finite subset $F\subseteq A$, there exists a unitary $v\in A$ such that $\|va-av\|<\varepsilon$ for all $a\in F$. If moreover $A$ is separable, then there exists a sequence $(u_n)_{n\in\N}$ of unitaries in $A$ with
$$\lim\limits_{n\to\infty} \|u_na-au_n\|=0$$
for all $a\in A$, such that $\pi_A((u_n)_{n\in\N})=u$ in $A_\infty$.
\end{lemma}

\begin{proposition}\label{prop: uniqueness}
Let $A$ and $B$ be C*-algebras and let $\alpha\colon G\to \Aut(A)$ and $\beta\colon G\to \Aut(B)$ be actions of a finite group $G$ such that $\beta$ has the Rokhlin property. Let $\phi,\psi\colon (A,\alpha) \to (B,\beta)$ be equivariant *-homomorphisms such that $\phi\sim_{\mathrm{au}}\psi$. Then $\phi\sim_{G-\mathrm{au}}\psi$.
\end{proposition}
\begin{proof}
Let $F$ be a finite subset of $A$ and let $\varepsilon>0$. We have to show that there exists a unitary $w\in \widetilde{B^\beta}$ such that
$$\|\phi(a)-w^*\psi(a)w\|<\varepsilon,$$
for all $a\in F$. Set $F'=\bigcup\limits_{g\in G}\alpha_g(F)$, which is again a finite subset of $A$. Since $\phi\sim_{\mathrm{au}}\psi$, there exists a unitary $u\in \widetilde{B}$ such that
\begin{equation}\label{au} \|\phi(b)-u^*\psi(b)u\|<\varepsilon\end{equation}
for all $b\in F'$. Choose $x\in B$ and $\lambda\in \C$ of modulus 1 such that $u=x+\lambda 1_{\widetilde{B}}$. Then equation (\ref{au}) above
is satisfied if one replaces $u$ with $\overline{\lambda}u$. Thus, we may assume that the unitary $u$ has the form $u=x+1_{\widetilde{B}}$ for some $x\in B$.

Fix $g\in G$ and $a\in F$. Then $b=\alpha_{g^{-1}}(a)$ belongs to $F'$. Using equation (\ref{au}) and the fact that
$\phi$ and $\psi$ are equivariant, we get
$$\|\beta_{g^{-1}}(\phi(a))-u^*\beta_{g^{-1}}(\psi(a))u\|<\varepsilon.$$
By applying $\beta_g$ to the inequality above, we conclude that
$$\|\phi(a)-\beta_g(u)^*\psi(a)\beta_g(u)\|<\varepsilon$$
for all $a\in F$ and $g\in G$

Choose positive orthogonal contractions $(r_g)_{g\in G}\subseteq B_\infty$ as in the definition of the Rokhlin property for $\beta$, and set $v=\sum\limits_{g\in G}\beta_g(x)r_g+1_{\widetilde{B}}$. Using that $x_g+1_{\widetilde{B}}$ is a unitary in $\widetilde{B}$, one checks that
\begin{align*}
& v^*v=\sum\limits_{g\in G} \left(\beta_g(x^*x)r_g^2+\beta_g(x)r_g+\beta_g(x)r_g\right)+1_{\widetilde{B}}=1_{\widetilde{B}}.
\end{align*}
Analogously, we have $vv^*=1_{\widetilde{B}}$, and hence $v$ is a unitary in $\widetilde{B}$. For every $b\in B$, we have
$$v^*bv=\sum\limits_{g\in G}r_g\beta_g(u)^*b\beta_g(u).$$
Therefore,
$$\|\phi(a)-v^*\psi(a)v\|=\left\|\sum\limits_{g\in G}r_g\phi(a)-\sum\limits_{g\in G}r_g\beta_g(u)^*\psi(a)\beta_g(u)\right\|<\varepsilon,$$
for all $a\in F$ (here we are considering $\phi$ and $\psi$ as maps from $A$ to $(\widetilde B)^\infty$, by composing them with the natural inclusion of $B$ in $(\widetilde B)^\infty$). Since $v=\sum\limits_{g\in G}\beta_g(xr_e)+1_{\widetilde{B}}$, we have $v\in (\widetilde{B^\beta})^\infty\subseteq(\widetilde{B})^\infty$. By Lemma \ref{lem: unitaries can be lifted}, we can choose a unitary $w\in \widetilde{B^\beta}$ such that
$$\|\phi(a)-w^*\psi(a)w\|<\varepsilon,$$
for all $a\in F$, and the proof is finished.
\end{proof}

\begin{lemma}\label{lem: limit of homomorphisms}
Let $A$ and $B$ be C*-algebras and let $\psi\colon A\to B$ be a *-homomorphism. Suppose there exists a sequence $(v_n)_{n\in\N}$ of unitaries in $\widetilde{B}$ such that the sequence $(v_n\phi(x)v_n^*)_{n\in\N}$ converges in $B$ for all $x$ in a dense subset of $A$. Then there exists a *-homomorphism $\psi\colon A\to B$ such that
$$\lim\limits_{n\to\infty} v_n\phi(x)v_n^*=\psi(x)$$
for all $x\in A$.
\end{lemma}
\begin{proof} Let
$$S=\{x\in A\colon (v_n\phi(x) v_n^*)_{n\in\N}\mbox{ converges in }B\}\subseteq A.$$
Then $S$ is a dense *-subalgebra of $A$. For each $x\in S$, denote by $\psi_0(x)$ the limit of the sequence $(v_n\phi(x)v_n^*)_{n\in\N}$. The map $\psi_0\colon S\to B$ is linear, multiplicative, preserves the adjoint operation, and is bounded by $\|\phi\|$, so it extends by continuity to a *-homomorphism $\psi \colon A\to B$. Given $a\in A$ and given $\varepsilon>0$, use density of $S$ in $A$ to choose $x\in S$ such that $\|a-x\|<\frac{\varepsilon}{3}$. Choose $N\in\N$ such that $\|v_N\phi(x)v_N^*-\psi(x)\|<\frac{\varepsilon}{3}$. Then
\begin{align*} \|\psi(a)-v_N\phi(a)v_N^*\|&\leq \|\psi(a-x)\|+\|\psi(x)-v_N\phi(x)v_N^*\|+\|v_N\phi(x)v_N^*-v_N\phi(a)v_N^*\|\\
&< \frac{\varepsilon}{3} + \frac{\varepsilon}{3}+\frac{\varepsilon}{3}=\varepsilon,\end{align*}
It follows that $\psi(a)=\lim\limits_{n\to\infty} v_n\phi(a)v_n^*$ for all $a\in A$, as desired.
\end{proof}

The unital case of the following proposition is \cite[Lemma 5.1]{Izumi-I}. Our proof for arbitrary C*-dynamical systems follows similar ideas.


\begin{proposition}\label{prop: existence}
Let $A$ and $B$ be C*-algebras and let $\alpha\colon G\to \Aut(A)$ and $\beta\colon G\to \Aut(B)$ be actions of a finite group $G$. Suppose that $A$ is separable and that $\beta$ has the Rokhlin property. Let $\phi\colon A\to B$ be a *-homomorphism such that $\beta_g\circ \phi\sim_{\mathrm{au}}\phi\circ \alpha_g$ for all $g\in G$. Then:
\begin{itemize}
\item[(i)] For any $\varepsilon>0$ and for any finite set $F\subseteq A$ there exists a unitary $u\in \widetilde{B}$ such that
\begin{align}\label{eq: phiphi}
\begin{aligned}
&\|(\beta_g\circ\Ad(w)\circ \phi)(x)-(\Ad(w)\circ\phi\circ \alpha_g)(x)\|<\varepsilon, \quad \forall g\in G,\, \forall x \in F,\\
&\|(\Ad(w)\circ\phi)(x)-\phi(x)\|<\varepsilon+\sup\limits_{g\in G}\|(\beta_g\circ\phi\circ\alpha_{g^{-1}})(x)-\phi(x)\|, \quad \forall x\in F.
\end{aligned}
\end{align}
\item[(ii)] There exists an equivariant *-homomorphism $\psi\colon A\to B$ that is approximately unitarily equivalent to $\phi$.
\end{itemize}
\end{proposition}
\begin{proof}
(i) Let $F$ be a finite subset of $A$ and let $\varepsilon>0$. Set $F'=\bigcup\limits_{g\in G}\alpha_g(F)$, which is a finite subset of $A$. Since $\beta_g\circ \phi\sim_{\mathrm{au}}\phi\circ\alpha_g$ for all $g\in G$, there exist unitaries $(u_g)_{g\in G}\subseteq \widetilde{B}$ such that
$$\left\|(\beta_g\circ\phi)(a)-(\Ad(u_g)\circ\phi\circ \alpha_g)(a)\right\|<\frac{\varepsilon}{2},$$
for all $a\in F'$ and $g\in G$. Upon replacing $u_g$ with a scalar multiple of it, one can assume that there are $(x_g)_{g\in G}\subseteq B$ such that $u_g=x_g+1_{\widetilde{B}}$ for all $g\in G$. For $a\in F$ and $g, h\in G$, we have
\begin{align*}
\|(\Ad(u_g)&\circ\phi\circ\alpha_h)(a)-(\beta_h\circ\Ad(u_{h^{-1}g})\circ\phi)(a)\| \\
&=\left\|(\Ad(u_g)\circ\phi\circ\alpha_g)(\alpha_{g^{-1}h}(a))-(\beta_h\circ\Ad(u_{h^{-1}g})\circ\phi\circ\alpha_{h^{-1}g})(\alpha_{g^{-1}h}(a))\right\|\\
&\le \left\|(\Ad(u_g)\circ\phi\circ\alpha_g)(\alpha_{g^{-1}h}(a))-(\beta_g\circ\phi)(\alpha_{g^{-1}h}(a))\right\|\\
& \ \ \ \ +\left\|(\beta_g\circ\phi)(\alpha_{g^{-1}h}(x))-(\beta_h\circ\Ad(u_{h^{-1}g})\circ\phi\circ\alpha_{h^{-1}g})(\alpha_{g^{-1}h}(x))\right\|\\
&\le \frac\varepsilon 2+\frac\varepsilon 2=\varepsilon.
\end{align*}
Choose positive orthogonal contractions $(r_g)_{g\in G}\subseteq B_\infty$ as in the definition of the Rokhlin property for $\beta$, and set
$$u=\sum\limits_{g\in G}r_gx_g+1_{\widetilde{B}}\in (\widetilde{B})^\infty.$$
Using that $x_g+1_{\widetilde{B}}$ is a unitary in $\widetilde{B}$, one checks that
$$u^*u=1_{\widetilde{B}}+\sum\limits_{g\in G}(r_g^2x_g^*x_g+r_gx_g+r_gx_g^*)=1_{\widetilde{B}}.$$
Analogously, one also checks that $uu^*=1_{\widetilde{B}}$, thus showing that $u$ is a unitary in $(\widetilde{B})^\infty$. The map $\Ad(u)$ can be written
in terms of the maps $\Ad(u_g)$ and the contractions $(r_g)_{g\in G}$, as follows:
$$(\Ad(u))(x)=uxu^*=\sum\limits_{g\in G}(u_gxu_g^*) r_g=\sum\limits_{g\in G}(\Ad(u_g))(x)r_g,$$
for all $x\in A$. Now for $a\in F$ and considering $\phi$ as a map from $A$ to $(\widetilde B)^\infty$ by composing it with the natural inclusion of $B$ in $(\widetilde B)^\infty$, we have the following identities
\begin{align*}
&(\beta_h\circ\Ad(u)\circ\phi)(a)=\sum\limits_{g\in G}r_{hg}(\beta_h\circ\Ad(u_g)\circ\phi)(a)=\sum\limits_{g\in G}r_g(\beta_h\circ\Ad(u_{h^{-1}g})\circ\phi)(a),\\
& (\Ad(u)\circ\phi\circ\alpha_h)(a)=\sum\limits_{g\in G}r_g(\Ad(u_g)\circ\phi\circ\alpha_h)(a).
\end{align*}
Therefore,
\begin{align*}
\|(\beta_h\circ\Ad(u)\circ\phi)(a) &- (\Ad(u)\circ\phi\circ\alpha_h)(a)\|\\
& \le \sup\limits_{g\in G}\left\|(\Ad(u_g)\circ\phi\circ\alpha_h)(a)-(\beta_h\circ\Ad(u_{h^{-1}g})\circ\phi)(a)\right\|<\varepsilon.\end{align*}
This in turn implies that
\begin{align*}
&\|(\Ad(u)\circ\phi)(a)-\phi(a)\|=\left\|\sum\limits_{g\in G}r_g((\Ad(u_g)\circ\phi)(a)-\phi(a))\right\|\\
&\le \sup\limits_{g\in G}\left\|(\Ad(u_g)\circ\phi)(a)-\phi(a)\right\|\\
&\le\sup\limits_{g\in G} \left(\left\|(\Ad(u_g)\circ\phi\circ\alpha_g)(\alpha_{g^{-1}}(a))-(\beta_g\circ\phi)(\alpha_{g^{-1}}(a))\right\|+\left\|(\beta_g\circ\phi\circ\alpha_{g^{-1}})(a)-\phi(a)\right\|\right)\\
&\le \varepsilon + \sup\limits_{g\in G}\left\|(\beta_g\circ\phi\circ\alpha_{g^{-1}})(a)-\phi(a)\right\|.
\end{align*}
We have shown that the inequalities in \eqref{eq: phiphi} hold for a unitary $u\in (\widetilde{B})^\infty$. By Lemma \ref{lem: unitaries can be lifted}, we can replace $u$ with a unitary in $w\in \widetilde{B}$ in such a way that both inequalities still hold for $w$ in place of $u$.


(ii) Let $(F_n)_{n\in \N}$ be an increasing sequence of finite subsets of $A$ whose union is dense in $A$. Upon replacing each $F_n$ with $\bigcup\limits_{g\in G}\alpha_g(F_n)$, we may assume that $\alpha_g(F_n)=F_n$ for all $g\in G$ and $n\in \N$. Set $\phi_1=\phi$ and find a unitary $u_1\in \widetilde{B}$ such that the conclusion of the first part of the proposition is satisfied with $\phi_1$ and $\varepsilon=1$. Set $\phi_2=\Ad(u_1)\circ \phi_1$, and find a unitary $u_2\in \widetilde{B}$ such that the conclusion of the first part of the proposition is satisfied with $\phi_2$ and $\varepsilon=\frac{1}{2}$. Iterating this process, there exist *-homomorphisms $\phi_n\colon A\to B$ with $\phi_1=\phi$ and unitaries $(u_n)_{n\in \N}$ in $\widetilde{B}$ such that $\phi_{n+1}=\Ad(u_n)\circ \phi_{n}$, for all $n\in \N$, which moreover for all $n\in \N$ satisfy
\begin{align*}
&\|(\beta_g\circ\phi_n)(x)-(\phi_n\circ\alpha_g)(x)\|<\frac{1}{2^n}\end{align*}
for all $g\in G$ and for all $x\in F_n$, and
\begin{align*}
\|\phi_{n+1}(x)-\phi_{n}(x)\|<\frac{3}{2^n}
\end{align*}
for all $x\in F_n$. For each $n\in \N$ set $v_n=u_n\cdots u_1$. Then the sequence of unitaries $(v_n)_{n\in\N}$ in $\widetilde{B}$ and the *-homomorphism
$\phi\colon A\to B$ satisfy the hypotheses of Lemma \ref{lem: limit of homomorphisms}, so it follows that the sequence $(\phi_n)_{n\in \N}$
converges to a *-homomorphism $\psi\colon A\to B$ that satisfies $\beta_g\circ\psi=\psi\circ\alpha_g$ for all $g\in G$; that is, $\psi$ is equivariant. Since each $\phi_n$ is unitarily equivalent to $\phi$, we conclude that $\phi$ and $\psi$ are approximately unitarily equivalent.
\end{proof}

\subsection{Categories of C*-dynamical systems and abstract classification}
Let $G$ be a second countable compact group and let $\mathbf A$ denote the category of separable C*-algebras. Let us denote by $\mathbf A_G $ the category whose objects are $G$-C*-dynamical systems $(A, \alpha)$, that is, $A$ is a C*-algebra and $\alpha\colon G\to \mathrm{Aut}(A)$ is a strongly continuous action, and whose morphisms are equivariant *-homomorphisms. We use the notation $\phi\colon (A,\alpha)\to (B,\beta)$ to denote equivariant *-homomorphisms $\phi\colon A\to B$. Approximate unitary equivalence of maps in this category is given in Definition \ref{df: B^G}.

If $\mathbf{B}$ is a subcategory of $\mathbf A$, we denote by $\mathbf B_G $ the full subcategory of $\mathbf A_G $ whose objects are C*-dynamical systems $(A,\alpha)$ with $A$ in $\mathbf B$, and whose morphisms are given by
$$\mathrm{Hom}_{\mathbf B_G }((A,\alpha),(B,\beta))=\mathrm{Hom}_{\mathbf A_G }((A,\alpha),(B,\beta)).$$

\begin{definition}
Let $\mathbf B$ be a subcategory of $\mathbf A$. Let $\mathrm F\colon \mathbf B_G \to \mathbf C$ be a functor from the category $\mathbf B_G$ to a category $\mathbf C$.
We say that the functor $\mathrm F$ {\it classifies homomorphisms} if:
\begin{itemize}
\item[(a)] For every pair of objects $(A,\alpha)$ and $(B,\beta)$ in $\mathbf B_G $ and for every morphism
$$\lambda\colon \mathrm F(A, \alpha)\to \mathrm F(B,\beta)$$
in $\mathbf C$, there exists a homomorphism $\phi\colon (A, \alpha)\to (B,\beta)$ in $\mathbf B_G $ such that $\mathrm F(\phi)=\lambda$.
\item[(b)] For every pair of objects $(A,\alpha)$ and $(B,\beta)$ in $\mathbf B_G $ and every pair of homomorphisms
$$\phi,\psi\colon (A, \alpha)\to (B,\beta),$$
one has $\mathrm F(\phi)=\mathrm F(\psi)$ if and only if $\phi\sim_{\mathrm{G-au}}\psi$.
\end{itemize}
We say that the functor $\mathrm F$ \emph{classifies isomorphisms} if it satisfies (a) and (b) above for ismorphisms instead of homomorphisms (such a functor is a \emph{strong classifying functor} in the sense of Elliott (see \cite{Elliott})).
\end{definition}


Let $\mathbf C_1$ and $\mathbf C_2$ be two categories. Recall that a functor $\mathrm F\colon \mathbf C_1\to \mathbf C_2$ is said to be \emph{sequentially continuous} if whenever $C=\varinjlim (C_n,\theta_n)$ in $\mathbf C_1$ for some sequential direct system $(C_n,\theta_n)_{n\in\N}$ in $\mathbf C_1$, then the inductive limit $\varinjlim (\mathrm F(C_n),\mathrm F(\theta_n))$ exists in $\mathbf C_2$, and one has
$$\mathrm F(\varinjlim (C_n,\theta_n))=\varinjlim (\mathrm F(C_n),\mathrm F(\theta_n)).$$

The following theorem is a consequence of \cite[Theorem 3]{Elliott}.

\begin{theorem}\label{thm: isomorphism}
Let $G$ be a second countable compact group, let $\mathbf B$ be a subcategory of $\mathbf A$, let $\mathbf B_G$ be the associated category of C*-dynamical systems, and
let $\mathbf C$ be a category in which inductive limits of sequences exist.
Let $\mathrm F\colon \mathbf B_G \to \mathbf C$ be a sequentially continuous functor that classifies homomorphisms. Then $\mathrm F$ classifies isomorphisms.
\end{theorem}
\begin{proof}
Let us briefly see that the conditions of \cite[Theorem 3]{Elliott} are satisfied for the category $\mathbf B_G$. First, using that the algebras in $\mathbf B_G$ are separable and that the group is second countable we can see that the set of equivariant *-homomorphisms between two C*-algebras in $\mathbf B_G$ is metrizable. Also, by taking the inner automorphisms of a C*-dynamical systems in $\mathbf B_G$ to be conjugation by unitaries in the unitization of the fixed point algebra of the given dynamical system, one can easily see that these automorphisms satisfy the conditions of \cite[Theorem 3]{Elliott}. Finally note that the category $\mathbf D$ whose objects are objects of $\mathbf C$ of the form $\mathrm F(A,\alpha)$ for some C*-dynamical system $(A, \alpha)$ in $\mathbf B_G$, and whose morphisms between two objects $\mathrm F(A,\alpha)$ and $\mathrm F(B,\alpha)$ are all the maps of the form $\mathrm F(\phi)$ for some equivariant *-homomorphism $\phi\colon (A, \alpha)\to (B, \beta)$, is just the classifying category of $B_G$ (in the sense of \cite{Elliott}), since $\mathrm F$ classifies homomorphisms by assumption. Therefore, by \cite[Theorem 3]{Elliott} the functor $\mathrm F$ is a strong classifying functor; in other words, it classifies *-isomorphisms.
\end{proof}

\begin{definition}\label{df: C^G, RB^G}
Let $G$ be a compact group. Let $\mathbf C$ be a category and let $\mathbf C_G $ denote the category whose objects are pairs $(C,\gamma)$,
where $C$ is an object in $\mathbf C$ and $\gamma\colon G\to \mathrm{Aut}(C)$ is a group homomorphism, also called an action of $G$ on $C$.
(We do not require any kind of continuity for this action since $C$ does not a priori have a topology.) The morphisms of $\mathbf C_G $ consist of
the morphisms of $\mathbf{C}$ that are equivariant.

Let $\mathbf B$ be a subcategory of $\mathbf A$ and let $\mathbf B_G$ be the associated category of C*-dynamical systems. Let $\mathrm F \colon \mathbf{B} \to \mathbf C$ be a functor.
Then $\mathrm F$ induces a functor $\mathrm F_G\colon \mathbf B_G \to \mathbf C_G $ as follows:
\begin{itemize}
\item[(i)] For an object $(A,\alpha)$ in $\mathbf B_G $, define an action $\mathrm F(\alpha)\colon G\to \mathrm{Aut}(\mathrm F(A))$ by
$(\mathrm F(\alpha))_g=\mathrm F(\alpha_g)$ for all $g\in G$. We then set $\mathrm F_G (A, \alpha)=(\mathrm F(A),\mathrm F(\alpha))$;
\item[(ii)] For a morphism $\phi\in \Hom_{\mathbf B_G }((A,\alpha),(B,\beta))$, we set $\mathrm F_G(\phi)= \mathrm F (\phi)$.
\end{itemize}
If $G$ is a finite group, we let $\mathbf{RB}_G$ denote the subcategory of $\mathbf B_G $ consisting of those C*-dynamical systems $(A,\alpha)$ in $\mathbf B_G $ with the Rokhlin property.
\end{definition}


The next theorem is a restatement, in the categorical setting, of Proposition \ref{prop: uniqueness} and Proposition \ref{prop: existence} (ii).

\begin{theorem}\label{thm: mainclassification}
Let $G$ be a finite group. Let $\mathbf B$, $\mathbf B_G$, $\mathbf{RB}_G$, $\mathbf C$, and $\mathbf C_G $ be as in Definition \ref{df: C^G, RB^G}. Let $\mathrm F\colon \mathbf B\to \mathbf C$ be a functor that classifies homomorphisms.
\begin{itemize}
\item[(i)] Let $(A, \alpha)$ be an object in $\mathbf B_G $ and let $(B,\beta)$ be an object in $\mathbf{RB}_G$.
\begin{itemize}
\item[(a)] For every morphism $\gamma\colon (\mathrm F(A), \mathrm F(\alpha))\to (\mathrm F(B), \mathrm F(\beta))$ in $\mathbf C_G $, there exists a morphism $\phi\colon (A, \alpha)\to(B, \beta)$ in $\mathbf B_G $ such that $\mathrm F_G(\phi)=\gamma$.
\item[(b)] If $\phi, \psi\colon (A, \alpha)\to(B, \beta)$ are morphisms in $\mathbf B_G $ such that $\mathrm F_G(\phi)=\mathrm F_G(\psi)$, then $\phi\sim_{G-\mathrm{au}}\psi$.
\end{itemize}
\item[(ii)] The restriction of the functor $\mathrm F_G$ to $\mathbf{RB}_G$ classifies homomorphisms.
\end{itemize}
\end{theorem}
\begin{proof}
(i) Let $(A, \alpha)$ be an object in $\mathbf B_G $ and let $(B,\beta)$ be an object in $\mathbf{RB}_G$.

(a) Let $\gamma\colon (\mathrm F(A), \mathrm F(\alpha))\to (\mathrm F(B), \mathrm F(\beta))$ be a morphism in $\mathbf C_G $. Using that $\mathrm F\colon \mathbf B\to \mathbf C$ classifies homomorphisms, choose a *-homomorphism $\psi\colon A\to B$ such that $\mathrm F(\psi)=\gamma$. Note that
$$\mathrm F(\psi\circ\alpha_g)=\mathrm F(\psi)\circ\mathrm F(\alpha_g)=\mathrm F(\beta_g)\circ \mathrm F(\psi)=\mathrm F(\beta_g\circ\psi),$$
for all $g\in G$. Using again that $\mathrm F$ classifies homomorphisms, we conclude that $\psi\circ\alpha_g$ and $\beta_g\circ\psi$ are approximately unitarily equivalent for all $g\in G$. Therefore, by part (ii) of Proposition \ref{prop: existence} there exists an equivariant *-homomorphism $\phi\colon (A,\alpha)\to (B,\beta)$ such that $\phi$ and $\psi$ are approximately unitarily equivalent. Thus $\phi$ is a morphism in $\mathbf B_G $ and
$$\mathrm F_G(\phi)=\mathrm F(\phi)=\mathrm F(\psi)=\gamma,$$
as desired.

(b) Let $\phi, \psi\colon (A,\alpha)\to(B,\beta)$ be morphisms in $\mathbf B_G $ such that $\mathrm F_G(\phi)=\mathrm F_G(\psi)$. Then $\phi\sim_{\mathrm{au}}\psi$ because $\mathrm F$ classifies homomorphisms and $\mathrm F$ agrees with $\mathrm F_G$ on morphisms. It then follows from Proposition \ref{prop: uniqueness} that $\phi\sim_{G-\mathrm{au}}\psi$.

Part (ii) clearly follows from (i).
\end{proof}

\begin{lemma}\label{lem: categories}
Let $G$ be a compact group, let $\Lambda$ be a directed set and let $\mathbf C$ be a category where inductive limits over $\Lambda$ exist.
Let $\mathbf C_G $ be the associated category as in Definition \ref{df: C^G, RB^G}. Then:
\begin{itemize}
\item[(i)] Inductive limits over $\Lambda$ exist in $\mathbf C_G $.
\item[(ii)] If $\mathbf D$ is a category where inductive limits over $\Lambda$ exist and $\mathrm F\colon \mathbf C\to \mathbf D$ is a functor that preserves direct limits over $\Lambda$, then the associated functor $\mathrm F_G\colon \mathbf C_G \to \mathbf D_G $ also preserves direct limits over $\Lambda$.
\end{itemize}
\end{lemma}
\begin{proof}
(i) Let $((C_\lambda,\alpha_\lambda)_{\lambda\in\Lambda}, (\gamma_{\lambda,\mu})_{\lambda,\mu\in\Lambda, \lambda<\mu})$ be a direct system in $\mathbf C_G $ over $\Lambda$, where $\gamma_{\lambda,\mu}\colon (C_\lambda,\alpha_\lambda)\to (C_\mu,\alpha_\mu, )$, for $\lambda<\mu$, is a morphism in $\mathbf C_G $. Let $(C,(\gamma_{\lambda,\infty})_{\lambda\in\Lambda})$, with $\gamma_{\lambda,\infty}\colon C_\lambda\to C$, be its direct limit in the category $\mathbf C$. Then
$$(\gamma_{\mu,\infty}\circ\alpha_\mu(g))\circ \gamma_{\lambda,\mu}=\gamma_{\lambda,\infty}\circ\alpha_\lambda(g)$$
for all $\mu\in \Lambda$ with $\lambda<\mu$. Hence, by the universal property of the inductive limit $(C,(\gamma_{\lambda,\infty})_{\lambda\in\Lambda})$, there exists a unique $\mathbf C$-morphism $\alpha(g)\colon C\to C$ that satisfies $\alpha(g)\circ \gamma_{\lambda,\infty}=\gamma_{\lambda,\infty}\circ \alpha_\lambda(g)$ for all $\lambda\in \Lambda$. Note that for $g,h\in G$, one has
$$(\alpha(g)\circ\alpha(h))\circ \gamma_{\lambda,\infty}= \gamma_{\lambda,\infty}\circ\alpha_\lambda(g)\circ\alpha(h)=\gamma_{\lambda,\infty}\circ\alpha_\lambda(gh)$$
for all $\lambda\in\Lambda$. By uniqueness of the morphism $\alpha(gh)$, it follows that $\alpha(g)\circ\alpha(h)=\alpha(gh)$ for all $g, h\in G$. This implies that $\alpha(g)$ is an automorphism of $C$ and that $\alpha\colon G\to \mathrm{Aut}(C)$ is an action. Thus $(C, \alpha)$ is an object in $\mathbf C_G $.

We claim that $(C, \alpha)$ is the inductive limit of $((C_\lambda,\alpha_\lambda)_{\lambda\in\Lambda}, (\gamma_{\lambda,\mu})_{\lambda,\mu\in\Lambda, \lambda<\mu})$ in the category $\mathbf C_G $. For $\lambda\in \Lambda$, The map $\gamma_{\lambda,\infty}$ is equivariant since $\gamma_{\lambda,\infty}\circ\alpha_\lambda(g)=\alpha(g)\circ\gamma_{\lambda,\infty}$ for all $g\in G$ and $\lambda\in\Lambda$. Let $(D, \beta)$ be an object in $\mathbf C_G $ and for $\lambda\in\Lambda$, let $\rho_\lambda\colon (C_\lambda,\alpha_\lambda)\to (D, \beta)$ be an equivariant morphism. By the universal property of the inductive limit $C$, there exists a unique morphism $\rho\colon C\to D$ satisfying $\rho_\lambda=\rho\circ \gamma_{\lambda,\infty}$ for all $\lambda\in\Lambda$. We therefore have
\begin{align*}
(\beta(g)^{-1}\circ \rho\circ\alpha(g))\circ \gamma_{\lambda,\infty}&=\beta^{-1}(g)\circ \rho\circ \gamma_{\lambda,\infty}\circ \alpha_\lambda(g)\\
&=\beta^{-1}(g)\circ\rho_{\lambda,\infty}\circ\alpha_\lambda(g)\\
&=\rho_{\lambda,\infty},
\end{align*}
for all $g\in G$ and $\lambda\in\Lambda$. Hence by uniqueness of $\rho$, we conclude that
$$\beta^{-1}(g)\circ \rho\circ \alpha(g)=\rho$$
for all $g\in G$. In other words, $\rho$ is equivariant. We have shown that $(C, \alpha)$ has the universal property of the inductive limit in $\mathbf C_G $, thus proving the claim and part (i).

(ii)  Let $((C_\lambda,\alpha_\lambda)_{\lambda\in\Lambda}, (\gamma_{\lambda,\mu})_{\lambda,\mu\in\Lambda, \lambda<\mu})$ be a direct system in $\mathbf C_G $ and let $(C, \alpha)$ be its inductive limit in $\mathbf C_G $, which exists by the first part of this lemma. We claim that $(\mathrm F(C), \mathrm F(\alpha))$ is the inductive limit of
$$\left((\mathrm F(C_\lambda),\mathrm F(\alpha_\lambda))_{\lambda\in\Lambda}, (\mathrm F(\gamma_{\lambda,\mu}))_{\lambda,\mu\in\Lambda, \lambda<\mu}\right)$$
in the category $\mathbf D_G $. Let $(D, \delta)$ be an object in $\mathbf D_G $ and for $\lambda\in\Lambda$, let
$\rho_\lambda\colon (\mathrm F(C_\lambda),\mathrm F(\alpha_\lambda))\to (D, \delta)$
be an equivariant morphism satisfying $\rho_\mu=\mathrm F(\gamma_{\lambda,\mu})\circ \rho_\lambda$ for all $\mu\in\Lambda$ with $\lambda<\mu$. Since $\mathrm F$ is continuous by assumption, we have
$$\mathrm F(C)=\varinjlim \left((\mathrm F(C_\lambda))_{\lambda\in\Lambda}, (\mathrm F(\gamma_{\lambda,\mu}))_{\lambda,\mu\in\Lambda, \lambda<\mu}\right)$$
in $\mathbf D$. By the universal property of the inductive limit $\mathrm F(C)$ in $\mathbf D$, there exits a unique morphism $\rho\colon \mathrm F(C)\to D$ in the category $\mathbf D$ satisfying $\rho\circ \mathrm F(\gamma_{\lambda,\infty})=\rho_\lambda$. It follows that
$$(\delta(g)^{-1}\circ \rho\circ \mathrm F(\alpha(g)))\circ \mathrm F(\gamma_{\lambda,\infty})=\delta(g)^{-1}\circ \rho_\lambda\circ \mathrm F(\alpha_\lambda(g))=\rho_\lambda,$$
for all $g\in G$ and $\lambda\in\Lambda$. By the uniqueness of the morphism $\rho$, we conclude that $\delta(g)^{-1}\circ \rho\circ \mathrm F(\alpha)(g)=\rho$ for all $g\in G$. That is, $\rho\colon (\mathrm F(C),\mathrm F(\alpha))\to (D,\delta)$ is equivariant. This shows that $(\mathrm F(C), \mathrm F(\alpha))$ has the universal property of inductive limits in $\mathbf D_G $.
\end{proof}


\begin{theorem}\label{thm: mainclassification1}
Let $G$ be a finite group, let $\mathbf B$ be a subcategory of $\mathbf A$, and let $\mathbf C$ be a category where inductive limits of sequences exist. Let $\mathbf{B}_G$, $\mathbf{RB}_G$, and $\mathbf C_G $ be as in Definition \ref{df: C^G, RB^G}. Let $\mathrm F\colon \mathbf{B}\to \mathbf C$ be a sequentially continuous functor that classifies homomorphisms and let $\mathrm{F}_G\colon \mathbf{B}_G\to \mathbf{C}_G$ be the associated functor as in Definition \ref{df: C^G, RB^G}. Then the restriction of $\mathrm{F}_G$ to $\mathbf{RB}_G$ classifies isomorphisms. In particular, if $(A,\alpha)$ and $(B,\beta)$ are C*-dynamical systems in $\mathbf{RB}_G$, then $\alpha$ and $\beta$ are conjugate if and only if there exists an isomorphism $\rho\colon \mathrm F_G(A,\alpha)\to \mathrm F_G(B,\beta)$ in $\mathbf C_G$.
\end{theorem}
\begin{proof}
Since by Theorem \ref{thm: mainclassification} the restriction of the functor $\mathrm{F}_G$ to $\mathbf{RB}_G$ classifies homomorphisms, it is sufficient to show that the conditions of Theorem \ref{thm: isomorphism} are satisfied. First note that sequential inductive limits exists in $\mathbf{B}_G$ since $G$ is finite and they exists in $\mathbf{B}$ by assumption. Now by \cite[Theorem 2 (v)]{SantiagoRP} the same is true for $\mathbf{RB}_G$. Given that sequential inductive limits exist in the category $\mathbf C$ and $\mathrm F\colon \mathbf{B}\to \mathbf C$ is sequentially continuous, it follows from Lemma \ref{lem: categories} applied to $\Lambda=\N$ that sequential inductive limits exist in $\mathbf{C}_G$ and the functor $\mathrm F_G\colon \mathbf{B}_G\to \mathbf{C}_G$ is sequentially continuous. In particular, it follows that the restriction of $\mathrm{F}_G$ to $\mathbf{RB}_G$ is sequentially continuous. This shows that the conditions of Theorem \ref{thm: isomorphism} are met. The last statement of the theorem follows from the definition of a functor that classifies isomorphisms.
\end{proof}

The following result was proved by Izumi in \cite[Theorem 3.5]{Izumi-I} for unital C*-algebras, and more recently by Nawata in \cite[Theorem 3.5]{Nawata} for C*-algebras with \emph{almost stable rank one} (that is, C*-algebras $A$ such that $A\subseteq\overline{\mathrm{GL}(\widetilde{A})}$).

\begin{theorem}\label{thm: isoexistence}
Let $G$ be a finite group, let $A$ be separable C*-algebra and let $\alpha$ and $\beta$ be actions of $G$ on $A$ with the Rokhlin property. Assume that $\alpha_g\sim_{\mathrm{au}}\beta_g$ for all $g\in G$. Then there exists an approximately inner automorphism $\psi$ of $A$ such that $\psi\circ\alpha_g=\beta_g\circ\psi$ for all $g\in G$.
\end{theorem}
\begin{proof}
Let $\mathbf C$ be the category whose objects are separable C*-algebras and whose morphisms are given by
$$\mathrm{Hom}(A, B)=\{[\phi]_{\mathrm{au}}\colon \phi\colon A\to B \mbox{ is a *-homomorphism}\},$$
where $[\phi]_{\mathrm{au}}$ denotes the approximate unitary equivalence class of $\phi$. (It is easy to check that composition of maps is well defined in $\mathbf{C}$, and thus $\mathbf{C}$ is indeed a category.) Let $\mathrm F\colon \mathbf{A}\to \mathbf{C}$ be the functor given by $\mathrm{F}(A)=A$ for any C*-algebra $A$ in $\mathbf A$, and $\mathrm F(\phi)=[\phi]_{\mathrm{au}}$ for any *-homomorphism $\phi$ in $\mathbf A$. It is straightforward to check that sequential inductive limits exist in $\mathbf C$ and that $\mathrm F$ is sequentially continuous. Moreover, by the construction of $\mathbf C$ and $\mathrm F$ it is clear that $\mathrm F$ classifies *-homomorphisms. Therefore, by Theorem \ref{thm: mainclassification1} the restriction of the associated functor $\mathrm{F}_G$ to $\mathbf{RA}_G$ classifies isomorphisms.

Let $A$ be a separable C*-algebra (that is, a C*-algebra in $\mathbf A$), and let $\alpha$ and $\beta$ be as in the statement of the theorem. Since $\alpha_g\sim_{\mathrm{au}}\beta_g$ for all $g\in G$, we have
$$\mathrm{F}(\mathrm{id_A})\circ\mathrm F(\alpha_g)=\mathrm{F}(\mathrm{id_A}\circ\alpha_g)=\mathrm F(\beta_g\circ\mathrm{id_A})=\mathrm F(\beta_g)\circ\mathrm{F}(\mathrm{id_A}),$$
for all $g\in G$.
In other words, the map $[\mathrm{id}_A]_{\mathrm{au}}$ is equivariant. Also, note that this map is an automorphism. Therefore, it is an isomorphism in the category $\mathbf{C}_G$. Since by the previous discussion, the restriction of $\mathrm{F}_G$ to $\mathbf{RA}_G$ classifies isomorphisms, it follows that that there exists an equivariant *-automorphism
$\psi\colon (A, \alpha)\to (A,\beta)$ such that $\mathrm{F}_G(\psi)=[\mathrm{id}_A]_{\mathrm{au}}$. In particular, $\alpha$ and $\beta$ are conjugate. Using that $\mathrm{F}(\psi)=\mathrm{F}_G(\psi)=[\mathrm{id}_A]_{\mathrm{au}}=\mathrm F(\mathrm{id}_A)$ and that $\mathrm F$ classifies homomorphisms, we get that $\psi\sim_{\mathrm{au}}\mathrm{id}_A$. In other words, $\psi$ is approximately inner.
\end{proof}

\begin{remark}
In view of \cite[Remark 3.6]{Nawata}, it may be worth pointing out that one can directly modify the proof of \cite[Lemma 3.4]{Nawata} to get rid of the assumption that $A$ has almost stable rank one. Indeed, one just needs to replace the element $w$ in the proof by the unitary $w'=\sum\limits_{g\in G}(v_g-\lambda_g 1_{\widetilde{A}})f_g+1_{\widetilde{A}}$, where $\lambda_g\in\C$ is such that $v_g-\lambda_g 1_{\widetilde{A}}\in A$.
\end{remark}

\subsection{Applications}
In this section we apply Theorems \ref{thm: mainclassification}, \ref{thm: mainclassification1}, and \ref{thm: isoexistence}, and known classification results, to obtain classification of equivariant *-homomorphisms and finite group actions on certain classes of 1-dimensional NCCW-complexes and AH-algebras.

\subsubsection{1-dimensional NCCW-complexes}
Let $E$ and $F$ be finite dimensional C*-algebras, and for $x\in [0,1]$, denote by $\mathrm{ev}_x\colon \mathrm{C}([0,1], F)\to F$ the
evaluation map at the point $x$. Recall that a C*-algebra $A$ is said to be a \emph{one-dimensional non-commutative CW-complex}, abbreviated 1-dimensional
NCCW-complex, if $A$ is given by a pullback diagram of the form:
\begin{equation*}
\xymatrix{A \ar[d] \ar[rr] & & E\ar[d] \\ \mathrm C([0,1], F)\ar[rr]_-{\mathrm{ev}_0\oplus \mathrm{ev}_1} & & F\oplus F.}
\end{equation*}

\begin{theorem}\label{thm: classification-h-RP-Cutilde}
Let $G$ be a finite group. Let $(A, \alpha)$ and $(B, \beta)$ be separable C*-dynamical systems such that $A$ can be written as an inductive limit of 1-dimensional NCCW-complexes with trivial $\mathrm K_1$-groups and such that $B$ has stable rank one. Assume that $\beta$ has the Rokhlin property.
\begin{itemize}
\item[(i)] Fix strictly positive elements $s_A$ and $s_B$ of $A$ and $B$, respectively. Let $\rho\colon \Cu^\sim(A)\to \Cu^\sim(B)$ be a morphism in the category $\CCu$ such that
    $$\rho([s_A])\le [s_B] \ \mbox{ and } \ \rho\circ\Cu^\sim(\alpha_g)=\Cu^\sim(\beta_g)\circ \rho$$
    for all $g\in G$. Then there exists an equivariant *-homomorphism
    $$\phi\colon (A, \alpha)\to (B, \beta) \ \mbox{ such that } \ \Cu^\sim(\phi)=\rho.$$
\item[(ii)] If $\phi, \psi\colon (A, \alpha)\to (B, \beta)$ are equivariant *-homomorphisms, then $\Cu^\sim(\phi)=\Cu^\sim(\psi)$ if and only if $\phi\sim_{G-\mathrm{au}}\psi$.
\end{itemize}

Moreover, if $A$ is unital, or if it is simple and has trivial $\mathrm K_0$-group, or if it can be written as an inductive limit of punctured-trees algebras, then the functor $\Cu^\sim$ can be replaced by the Cuntz functor $\Cu$ in the statement of this theorem.
\end{theorem}
\begin{proof}
(i) Let $\rho\colon \Cu^\sim(A)\to \Cu^\sim(B)$ be as in the statement of the theorem. By \cite[Theorem 1]{Robert}, there exists a *-homomorphism $\psi\colon A\to B$ such that $\Cu^\sim(\psi)=\rho$. Using that $\rho$ is equivariant, we get $\Cu^\sim(\beta_g\circ\psi)=\Cu^\sim(\psi\circ\alpha_g)$ for all $g\in G$. By the uniqueness part of \cite[Theorem 1]{Robert}, it follows that $\beta_g\circ\psi\sim_{\mathrm{au}}\psi\circ\alpha_g$ for all $g\in G$. By Proposition \ref{prop: existence}, there exists an equivariant *-homomorphism $\phi\colon A\to B$ such that $\phi\sim_{\mathrm{au}}\psi$. Since $\Cu^\sim$ is invariant under approximate unitary equivalence,
we conclude $\Cu^\sim(\phi)=\Cu^\sim(\psi)$, as desired.

(ii) The ``if" implication is clear. For the converse, let $\phi$ and $\psi$
be as in the statement of the theorem. By the uniqueness part of \cite[Theorem 1]{Robert}, we have $\phi\sim_{\mathrm{au}}\psi$. It now follows from Proposition \ref{prop: uniqueness} that $\phi\sim_{G-\mathrm{au}}\psi$.

It follows from \cite[Remark 3 (ii)]{Robert}, and by \cite[Corollary 4 (b)]{Robert}, \cite[Corollary 6.7]{Elliott-Robert-Santiago}, and \cite[Corollary 8.6]{Tikuisis}, respectively, that the functors $\Cu^\sim$ and $\Cu$ are equivalent when restricted to the class of C*-algebras that are inductive limits of 1-dimensional NCCW-complexes which are either unital or simple and with trivial $\mathrm K_0$-group. Hence, for these classes of C*-algebras, the theorem holds when $\Cu^\sim$ is replaced by $\Cu$. For C*-algebras that are inductive limits of punctured-trees algebras, one can use \cite[Theorem 1.1]{Ciuperca-Elliott-Santiago} instead of \cite[Theorem 1]{Robert} in the proof above to obtain the desired result. Finally, since $B$ has stable rank one, the results in
\cite{Robert} show that $\Cu(B)$ is a subsemigroup of $\Cu^\sim(B)$. In particular, for a homomorphism
$\phi\colon A\to B$, the range of the induced map $\Cu^\sim(\phi)\colon \Cu^\sim(A)\to \Cu^\sim(B)$ is
contained in $\Cu(B)\subseteq \Cu^\sim(B)$.
\end{proof}

\begin{theorem}\label{classif Rp on NCCW}
Let $G$ be a finite group, and let $(A, \alpha)$ and $(B, \beta)$ be separable dynamical systems such that $A$ and $B$ can be written as inductive limits of 1-dimensional NCCW-complexes with trivial $\mathrm K_1$-groups. Suppose that $\alpha$ and $\beta$ have the Rokhlin property.
\begin{itemize}
\item[(i)] Fix strictly positive elements $s_A$ and $s_B$ of $A$ and $B$ respectively. Then the actions $\alpha$ and $\beta$ are conjugate if and only if there exists an isomorphism $\gamma\colon \Cu^\sim(A)\to \Cu^\sim(B)$ with $\gamma([s_A])=[s_B]$, such that
    $$\gamma\circ\Cu^\sim(\alpha_g)=\Cu^\sim(\beta_g)\circ\gamma \mbox{ for all } g\in G.$$
\item[(ii)] Assume that $A=B$. Then the actions $\alpha$ and $\beta$ are conjugate by an approximately inner automorphism of $A$ if and only if $\Cu^\sim(\alpha_g)=\Cu^\sim(\beta_g)$ for all $g\in G$.
\end{itemize}

Moreover, if both $A$ and $B$ are unital, or if they are simple and have trivial $\mathrm K_0$-groups, or if they can be written as inductive limits of punctured-trees algebras, then the functor $\Cu^\sim$ can be replaced by the Cuntz functor $\Cu$.
\end{theorem}
\begin{proof}
Part (ii) clearly follows from (i). Let us prove (i).
Let $\mathbf B$ denote the subcategory of the category $\mathbf A$ of C*-algebras consisting of those C*-algebras that can be written as an inductive limit of 1-dimensional NCCW-complexes with trivial $\mathrm K_1$-groups. By \cite[Theorem 1]{Robert}, the functor $(\Cu^\sim(\cdot), [s_{\,\cdot\,}])$, where $s_{\,\cdot\,}$ is a strictly positive element of the given algebra, restricted to $\mathbf B$ classifies homomorphisms. Therefore, by Theorem \ref{thm: mainclassification1}, the associated functor $(\Cu^\sim_G(\cdot), [s_{\,\cdot\,}])$ restricted to $\mathbf{RB}_G$ classifies isomorphisms, which implies (i).

The last part of the theorem follows from the same arguments used at the end of the proof of Theorem \ref{thm: classification-h-RP-Cutilde}.
\end{proof}

Let $G$ be a finite group. Recall that the action $\mu^G\colon G\to \Aut\left(\M_{|G|^\infty}\right)$ constructed in Example \ref{eg: model action} has the Rokhlin property, and that $\mu^G_g$ is approximately inner for all $g\in G$.

In the next corollary, we do not assume that either $\alpha$ or $\beta$ has the Rokhlin property.

\begin{corollary}\label{cor: tensor with mu-G}
Let $G$ be a finite group and let $(A, \alpha)$ and $(A, \beta)$ be C*-dynamical systems such that $A$ can be written as an inductive limit of 1-dimensional NCCW-complexes with trivial $\mathrm K_1$-groups. Suppose that $\Cu^\sim(\alpha_g)=\Cu^\sim(\beta_g)$ for all $g\in G$. Then $\alpha\otimes \mu^G$ and $\beta\otimes \mu^{G}$ are conjugate.

Moreover, if $A$ belongs to one of the classes of C*-algebras described in the last part of Theorem \ref{classif Rp on NCCW}, then the statement of the corollary holds for the functor $\Cu$ in place of the functor $\Cu^\sim$.
\end{corollary}
\begin{proof}
The actions $\alpha\otimes \mu^G$ and $\beta\otimes \mu^{G}$ have the Rokhlin property by part (i) of Lemma \ref{Rp properties}. Note that $\mu^G_g$ is approximately inner for all $g\in G$. Thus,
$$\Cu^\sim(\alpha\otimes\mu^G_g)=\Cu^\sim(\alpha\otimes \id_{\mathrm M_{|G|^\infty}})=\Cu^\sim(\beta\otimes \id_{\mathrm M_{|G|^\infty}})=\Cu^\sim(\beta\otimes\mu^G_g)$$
for all $g\in G$. It follows from Theorem \ref{classif Rp on NCCW} (ii) that $\alpha\otimes \mu^G$ and $\beta\otimes \mu^{G}$ are conjugate.
\end{proof}

\subsubsection{AH-algebras}

Recall that a C*-algebra $A$ is approximate homogeneous (shortly AH) if it can be written as an inductive limit $A=\varinjlim (A_n, \phi_{n,m})$, with
$$A_n=\oplus_{j=1}^{s(n)}P_{n,j}\M_{n,j}(\mathrm{C}(X_{n,j}))P_{n,j},$$
where $X_{n,j}$ is a finite dimensional compact metric space, and $P_{n,j}\in \M_{n,j}(\mathrm{C}(X_{n,j}))$ is a projection for all $n$ and $j$. The C*-algebra $A$ is said to have \emph{no dimension growth} if there exists an inductive limit decomposition of $A$ as an AH-algebra such that
$$\sup_n\max_j\dim X_{n,j}<\infty.$$

Let $A$ be a unital simple separable C*-algebra and let $\mathrm T(A)$ denote the metrizable compact convex set of tracial states of $A$. Denote by $\mathrm T$ the induced contravariant functor from the category of unital separable simple C*-algebras to the category of metrizable compact convex sets. It is not difficult to check that $\mathrm T$ is continuous, meaning that it sends inductive limits to projective limits. 

Let $T$ be a metrizable compact convex set and let $\mathrm{Aff}(T)$ denote the set of real-valued continuous affine functions on $T$. Let $\mathrm{Aff}$ denote the induced contravariant functor from the category of metrizable compact convex sets to the category of normed vector spaces. Denote by $\rho_A\colon \mathrm{K}_0(A)\to \mathrm{Aff}(\mathrm T(A))$ the map defined by
\begin{align}\label{eq: rho}
\rho_A([p]-[q])(\tau)=(\tau\otimes \mathrm{Tr}_n)(p)-(\tau\otimes \mathrm{Tr}_n)(q)
\end{align}
for $p, q\in \M_n(\C)$, where $\mathrm{Tr}_n$ denotes the standard trace on $\M_n(\C)$.

Let $A$ be a unital C*-algebra. Denote by $\mathrm U(A)$ the unitary group of $A$ and by $\mathrm{CU}(A)$ the closure of the normal subgroup generated by the commutators of $\mathrm U(A)$. We denote the quotient group by
$$\mathrm H(A)=\mathrm U(A)/\mathrm{CU}(A).$$
(see \cite{Thomsen} and \cite{Nielsen-Thomsen} for properties of this group).
The set $\mathrm H(A)$, endowed with the distance induced by the distance in $\mathrm U(A)$, is a complete metric space. 
We denote by $\mathrm H$ the induced functor from the category of C*-algebras to the category of complete metric groups. Also, if $A$ is a simple unital AH-algebra of no dimension growth (or more generally, a simple unital C*-algebra of tracial rank no greater than one), then there exists an injection
\begin{align}\label{eq: lambda}
\lambda_A\colon \mathrm{Aff}(\mathrm T(A))/\overline{\rho_A(\mathrm K_0(A))}\to \mathrm H(A).
\end{align}
(See \cite{Thomsen} and \cite{Huaxin}.)

For a C*-algebra $A$, denote by $\underline{\mathrm K}(A)$ the sum of all K-groups with $\mathbb{Z}/n\mathbb{Z}$ coefficients for all $n\ge 1$. Let $\Lambda$ denote the category generated by the Bockstein operations on $\underline{\mathrm K}(A)$ (see \cite{Dadarlat-Loring}). Then $\underline{\mathrm K}(A)$ becomes a $\Lambda$-module and it induces a continuous functor $\underline{\mathrm K}$ from the category of C*-algebras to the category of $\Lambda$-modules.

Let $A$ and $B$ be unital simple AH-algebras and let $\mathrm{KL}(A, B)$ denote the group defined in \cite{RordamKL}. By the Universal Coefficient Theorem and the Universal Multicoefficient Theorem (see \cite{Dadarlat-Loring}), the groups $\mathrm{KL}(A,B)$ and $\mathrm{Hom}_{\Lambda}(\underline{\mathrm K}(A),\underline{\mathrm K}(B))$ are naturally isomorphic. Let $\mathrm{KL}^{++}_e(A, B)$ be as in \cite[Definition 6.4]{Huaxin}. By the previous isomorphism, the group $\mathrm{KL}^{++}_e(A,B)$ is naturally isomorphic to
$$\{\kappa\in \mathrm{Hom}_{\Lambda}(\underline{\mathrm K}(A),\underline{\mathrm K}(B))\colon \kappa(\mathrm{K}_0(A)_+\setminus \{0\})\subseteq \mathrm{K}_0(B)_+\setminus \{0\},\, \kappa([1_A])=\kappa([1_B])\}.$$

Let us define a functor $\underline{\mathrm K}^{++}$ from the category of separable, unital, simple, finite C*-algebras to the category whose objects are 4-tuples $(M, N, E, e)$, where $M$ is a $\Lambda$-module, $N$ is a subgroup of $M$, $E$ is a subset of $N$, and $e$ is an element of $N$; and whose morphisms $\kappa\colon (M, N, E, e)\to (M', N', E', e')$ are $\Lambda$-module maps $\kappa\colon M\to M'$ such that $\kappa(N)\subseteq N'$, $\kappa(E)\subseteq E'$, and $\kappa(e)=e'$. The functor $\underline{\mathrm K}^{++}$ is defined as follows:
$$\underline{\mathrm K}^{++}(A)=(\underline{\mathrm K}(A), \mathrm K_0(A), \mathrm K_0(A)_+\setminus \{0\}, [1_A]), \quad\mbox{ and }\quad \underline{\mathrm K}^{++}(\phi)=\underline{\mathrm K}(\phi).$$
Note that if $A$ and $B$ are unital AH-algebras then $\mathrm{KL}^{++}_e(A,B)$ is isomorphic to $\mathrm{Hom}(\underline{\mathrm K}^{++}(A), \underline{\mathrm K}^{++}(B))$.

Let $\mathbf C$ denote the category whose objects are tuples
$$\left((M, N, E, e), T, H, \rho, \lambda\right),$$
where $(M, N, E, e)$ is as above, $T$ is a metrizable compact convex set, $H$ is a complete metric group, $\rho\colon N\to \mathrm{Aff}(T)$ is a group homomorphism, and $\lambda\colon \mathrm{Aff}(T)/\overline{\rho(N)}\to H$ is an injective continuous group homomorphism. The maps in $\mathbf C$ are triples
$$(\kappa, \eta, \mu)\colon \left((M, N, E, e), T, H, \rho, \lambda\right)\to \left((M', N', E', e'), T', H', \rho', \lambda'\right),$$
where $\kappa\colon (M, N, E, e)\to (M', N', E', e')$, $\eta\colon T'\to T$, and $\mu\colon H\to H'$ are maps in the corresponding categories that satisfy the compatibility conditions:
$$\rho'\circ\kappa|_N=\mathrm{Aff}(\eta)\circ \rho, \quad \mbox{and}\quad \lambda'\circ \mu=\overline{\mathrm{Aff}}(\eta)\circ\lambda,$$
where
$$\overline{\mathrm{Aff}}(\eta)\colon \mathrm{Aff}(T)/\overline{\rho(N)}\to \mathrm{Aff}(T')/\overline{\rho(N')}$$
is the map induced by $\mathrm{Aff}(\eta)$.
Using that inductive limits of sequences exist in each of the categories that form $\mathbf C$, it is not difficult to show that $\mathbf C$ is also closed under taking inductive limits of sequences. Also, it is easy to see that $\mathrm F=(\underline{\mathrm K}^{++}, \mathrm T,\mathrm H)$ is a functor from the category of unital, simple, separable, finite C*-algebras to the category $\mathbf C$. Moreover, since the functors that form $\mathrm F$ are continuous, $\mathrm F$ is also continuous.

\begin{theorem}\label{thm: AH homomrphisms}
Let $G$ be a finite group. Let $(A, \alpha)$ and $(B, \beta)$ be dynamical systems such that $A$ and $B$ are unital simple AH-algebras of no dimension growth. Assume that $\beta$ has the Rokhlin property.
\begin{itemize}
\item[(i)] Let
\begin{align*}
\kappa\colon \underline{\mathrm K}^{++}(A)\to \underline{\mathrm K}^{++}(B),\quad \eta\colon \mathrm T(B)\to \mathrm T(A),\quad \mbox{and}\quad
\mu\colon \mathrm H(A)\to\mathrm H(B),
\end{align*}
be maps in the corresponding categories that satisfy the compatibility conditions
\begin{align*}
\rho_B\circ\kappa|_{\mathrm{K_0(A)}}=\mathrm{Aff}(\eta)\circ \rho_A, \quad \mbox{and}\quad \lambda_B\circ \mu=\overline{\mathrm{Aff}}(\eta)\circ\lambda_A,
\end{align*}
where $\rho_A$, $\rho_B$, $\lambda_A$, and $\lambda_B$ are as in \eqref{eq: rho} and \eqref{eq: lambda}. Suppose that
\begin{align*}
\kappa\circ\underline{\mathrm K}(\alpha_g)= \underline{\mathrm K}(\beta_g)\circ \kappa,
\quad\eta\circ\mathrm T(\beta_g)=\mathrm T(\alpha_g)\circ \eta,\quad \mu\circ \mathrm H(\alpha_g)=\mathrm H(\beta_g)\circ \mu,
\end{align*}
for all $g\in G$. Then there exists an equivariant *-homomorphism $\phi\colon (A, \alpha)\to (B, \beta)$ such that
$$\underline{\mathrm K}^{++}(\phi)=\kappa, \quad \mathrm T(\phi)=\eta,\quad \mbox{and}\quad \mathrm H(\phi)=\mu.$$

\item[(ii)] Let $\phi, \psi\colon A\to B$ be equivariant *-homomorphisms such that
$$\underline{\mathrm K}(\phi)=\underline{\mathrm K}(\psi),\quad \mathrm T(\phi)=\mathrm T(\psi), \quad \mbox{and}\quad \mathrm H(\phi)=\mathrm H(\psi).$$
Then $\phi\sim_{G-\mathrm{au}}\psi$.
\end{itemize}
\end{theorem}
\begin{proof}
It is shown in \cite{Gong} that every unital simple AH-algebra of no dimension growth has tracial rank almost one. By \cite[Theorems 5.11 and 6.10]{Huaxin} applied to the algebras $A$ and $B$, and using the computations of $\mathrm{KL}^{++}_e(A,B)$ given in the paragraphs preceding the theorem, we deduce that the functor $\mathrm F$ (defined above) restricted to the category of unital simple AH-algebras of no dimension growth classifies *-homomorphisms. The theorem now follows from Theorem \ref{thm: mainclassification}.
\end{proof}

\begin{theorem}\label{thm: ismomorphismAH}
Let $G$ be a finite group and let $A$ and $B$ be unital simple AH-algebras of no dimension growth. Let $\alpha$ and $\beta$ be actions of $G$ on $A$ and $B$ with the Rokhlin property.
\begin{itemize}
\item[(i)]
The actions $\alpha$ and $\beta$ are conjugate if and only if there exist isomorphisms
\begin{align*}
\kappa\colon \underline{\mathrm K}^{++}(A)\to \underline{\mathrm K}^{++}(B), \quad\eta\colon \mathrm T(B)\to \mathrm T(A),\quad
&\mu\colon \mathrm H(A)\to\mathrm H(B),
\end{align*}
in the corresponding categories, that satisfy the compatibility conditions of the previous theorem, and such that
\begin{align*}
\kappa\circ\underline{\mathrm K}(\alpha_g)= \underline{\mathrm K}(\beta_g)\circ \kappa,\quad
\eta\circ\mathrm T(\beta_g)=\mathrm T(\alpha_g)\circ \rho,\quad \mu\circ \mathrm H(\alpha_g)=\mathrm H(\beta_g)\circ \lambda,
\end{align*}
for all $g\in G$.

\item[(ii)] Assume that $A=B$. Then the actions $\alpha$ and $\beta$ are conjugate by an approximately inner automorphism if and only if
\begin{align*}
\underline{\mathrm K}(\alpha_g)=\underline{\mathrm K}(\beta_g),\quad \mathrm T(\alpha_g)=\mathrm T(\beta_g), \quad \mbox{and}\quad \mathrm H(\alpha_g)=\mathrm H(\beta_g),
\end{align*}
for all $g\in G$.
\end{itemize}
\end{theorem}
\begin{proof}
Part (ii) clearly follows from (i) and part (ii) of Theorem~\ref{thm: AH homomrphisms}. 
Let us prove (i). As in the proof of Theorem \ref{thm: AH homomrphisms}, the functor $\mathrm F$ restricted to the category of unital simple AH-algebras of no dimension growth classifies homomorphisms. The statements of the theorem now follows from Theorem \ref{thm: mainclassification1}.
\end{proof}

\begin{corollary}
Let $G$ be a finite group and let $A$ be a unital simple AH-algebra of no dimension growth. Let $(A, \alpha)$ and $(A, \beta)$ be C*-dynamical systems. Suppose that
$$\underline{\mathrm K}^{++}(\alpha_g)=\underline{\mathrm K}^{++}(\beta_g), \quad \mathrm T(\alpha_g)=\mathrm T(\beta_g), \quad \mbox{and} \quad \mathrm H(\alpha_g)=\mathrm{H}(\beta_g),$$
for all $g\in G$. Then $\alpha\otimes \mu^G$ and $\beta\otimes \mu^{G}$ are conjugate.
\end{corollary}
\begin{proof}
The proof of this corollary follows line by line the proof of Corollary \ref{cor: tensor with mu-G}, using the functor $\mathrm F$ instead of the functor $\Cu^\sim$ and Theorem \ref{thm: ismomorphismAH} instead of Theorem \ref{classif Rp on NCCW}. 
\end{proof}

\section{Cuntz semigroup and K-theoretical constraints}
In this section, a Cuntz semigroup obstruction is obtained for a C*-algebra to admit an action with the Rokhlin property. Also, the Cuntz semigroup of the fixed-point C*-algebra and the crossed product C*-algebra associated to an action of a finite group with the Rokhlin property are computed in terms of the Cuntz semigroup of the given algebra. As a corollary, similar results are obtained for the Murray-von Neumann semigroup and the K-groups.

Let $G$ be a group and let $(S,\gamma)$ be an object in the category $\CCu_G$, this is, $S$ is a semigroup in the category $\CCu$ and
$\gamma\colon G\to \Aut(S)$ is an action of $G$ on $S$. Let $S^\gamma$ and $S^\gamma_\N$ be the subsemigroups of $S$ defined in Definition \ref{df: S gamma}.
It was shown in Lemma \ref{lem: closure} that $S^\gamma$ belongs to the category $\CCu$ and that $S^\gamma_\N$ is closed under suprema
of increasing sequences. We do not know in general whether $S^\gamma_\N$ is an object in $\CCu$. However, if $\alpha\colon G\to\Aut(A)$ is
an action of a finite group $G$ on a C*-algebra $A$ with the Rokhlin property, it will follow by the next theorem that
$\Cu(A)_\N^{\Cu(\alpha)}$ coincides with $\Cu(A)^{\Cu(\alpha)}$, and with the Cuntz semigroup of $A^\alpha$, so in
particular belongs to $\CCu$.

For use in the proof of the next theorem, if $\phi\colon A\to B$ is a *-homomorphism between C*-algebras $A$ and $B$, we denote by $\phi^s\colon A\otimes\K \to B\otimes \K$ the stabilized *-homomorphism $\phi^s=\phi\otimes\id_\K$.

\begin{theorem}\label{thm: Rokhlin Constraint}
Let $A$ be a C*-algebra and let $\alpha$ be an action of a finite group $G$ on $A$ with the Rokhlin property. Let $i\colon A^\alpha\to A$
be the inclusion map. Then:
\begin{itemize}
\item[(i)] The map $\Cu(\widetilde i)\colon\Cu(\widetilde{A^\alpha})\to \Cu(\widetilde A)$ is an order embedding;

\item[(ii)] The map $\Cu(i)\colon\Cu(A^\alpha)\to \Cu(A)$ is an order embedding and
\begin{align*}
\begin{aligned}
\mathrm{Im}(\Cu(i))&=\overline{\mathrm{Im}\left(\sum\limits_{g\in G}\Cu(\alpha_g)\right)}=\mathrm{\Cu(A)}_\N^{\Cu(\alpha)}
=\mathrm{\Cu(A)}^{\Cu(\alpha)};
\end{aligned}
\end{align*}
\end{itemize}
\end{theorem}
\begin{proof}
In the proof of this theorem, we will denote the action induced by $\alpha$ on $\widetilde{A}\otimes \K$ again
by $\alpha$.

(i) Let $a,b\in \widetilde{A^\alpha}\otimes \K$ satisfy $a\precsim b$ in $\widetilde A\otimes \K$. We want to
show that $a\precsim b$ in $\widetilde{A^\alpha}\otimes \K$. Let $\varepsilon>0$. By Lemma \ref{lem: Cuntz relation}, there exists $d\in \widetilde A\otimes \K$ such that $(a-\varepsilon)_+=dbd^*$. Apply $\alpha_g$ to this equation to get $(a-\varepsilon)_+=\alpha_g(d)b\alpha_g(d^*)$ for all $g\in G$.

Let $\pi\colon \widetilde A\to \C$ be the quotient map and let $j\colon \C\to \widetilde A$ be the inclusion $j(\lambda)=\lambda 1_{\widetilde A}$ for all $\lambda\in \C$. It is clear that $\pi\circ j=\id_{\C}$. Set
\begin{align*}
& a_1=(j^s\circ \pi^s)((a-\varepsilon)_+)\in \C 1_{\widetilde A}\otimes \K, \qquad a_2=(a-\varepsilon)_+-a_1\in A^\alpha\otimes \K,\\
& b_1=(j^s\circ \pi^s)(b)\in \C 1_{\widetilde A}\otimes \K, \qquad \qquad\quad \, b_2=b-b_1\in A^\alpha\otimes \K,\\
& d_1=(j^s\circ \pi^s)(d)\in\C 1_{\widetilde A}\otimes \K, \qquad \qquad\quad  d_2=d-d_1\in A\otimes \K.
\end{align*}
Then $a_1=d_1b_1d_1^*$. Set
$$F=\{\alpha_g(d_2)b\alpha_g(d_2)\colon g\in G\}\cup \{d_1b\alpha_g(d_2^*)\colon g\in G\}\cup \{a_2-d_1b_2d_1^*\}\subseteq
\widetilde{A}\otimes\K.$$
Use Lemma \ref{lem: Rokhlin equivalence} (ii) to choose orthogonal positive contractions $(r_g)_{g\in G}$ in
$(A\otimes \K)^\infty\cap F'\subseteq (\widetilde A\otimes \K)^\infty\cap F'$ such that $\alpha_g(r_g)=r_{gh}$ for all $g,h\in G$, and $(\sum\limits_{g\in G}r_g)x=x$ for all $x\in F$. Set
$$f=\sum\limits_{g\in G}r_g\alpha_g(d_2)+d_1\in (\widetilde A\otimes \K)^\infty.$$

In the following computation, we use in the first step the identities $r_gx=xr_g$ for all $g\in G$ and $x\in F$,
$r_gr_h=0$ for all $g\neq h$, and $(r_g^2-r_g)x=x$ for all $g\in G$ and $x\in F$;
in the second step the definition of $d_2$;
in the fourth step that $d_1\in (\widetilde A\otimes \K)^\alpha$ and the identity $(a-\varepsilon)_+=\alpha_g(d)b\alpha_g(d^*)$ for all $g\in G$;
in the fifth step the identity $a_1=d_1b_1d_1^*$;
and in the last step the identity $(\sum\limits_{g\in G}r_g)x=x$ for all $x\in F$:
\begin{align*}
&fbf^*=\left(\sum\limits_{g,h\in G}r_g\alpha_g(d_2)b\alpha_h(d_2^*)r_h+\sum\limits_{g\in G}r_g\alpha_g(d_2)bd_1^*+\sum\limits_{g\in G}d_1b\alpha_g(d_2^*)r_g\right)+d_1bd_1^*\\
&=\left(\sum\limits_{g\in G}r_g\alpha_g(d_2)b\alpha_g(d_2^*)+\sum\limits_{g\in G}r_g\alpha_g(d_2)bd_1^*+\sum\limits_{g\in G}r_gd_1b\alpha_g(d_2^*)\right)+d_1bd_1^*\\
&=\left(\sum\limits_{g\in G}r_g\left(\alpha_g(d-d_1)b\alpha_g(d^*-d_1^*)+\alpha_g(d_2)bd_1^*+d_1b\alpha_g(d_2^*)\right)\right)+d_1bd_1^*\\
&=\left(\sum\limits_{g\in G}r_g\left(\alpha_g(d)b\alpha_g(d^*)-\alpha_g(d-d_2)bd_1^*-d_1b\alpha_g(d^*-d_2^*)+d_1bd_1^*\right)\right)+d_1bd_1^*\\
&=\left(\sum\limits_{g\in G}r_g\left((a-\varepsilon)_+-d_1bd_1^*-d_1bd_1^*+d_1bd_1^*\right)\right)+d_1bd_1^*\\
&=\left(\sum\limits_{g\in G}r_g\left((a-\varepsilon)_+-d_1bd_1^*\right)\right)+d_1bd_1^*\\
&=\left(\sum\limits_{g\in G}r_g\left(a_1+a_2-d_1b_1d_1^*-d_1b_2d_1^*\right)\right)+d_1bd_1^*\\
&=\left(\sum\limits_{g\in G}r_g\left(a_2-d_1b_2d_1^*\right)\right)+d_1bd_1^*\\
&=a_2+d_1b_1d_1^*\\
&=(a-\varepsilon)_+.
\end{align*}
Shortly, $(a-\varepsilon)_+=fbf^*$ in $(\widetilde A\otimes \K)^\infty$. Since
$$f=\sum\limits_{g\in G}r_g\alpha_g(d_2)+d_1=\sum\limits_{g\in G}\alpha_g(r_ed_2)+d_1,$$
it follows that $\alpha_g(f)=f$ for all $g\in G$. This implies that $f$ is the image of a sequence $(f_n)_{n\in\N}$ in $\ell^\infty(\N,\widetilde{A^\alpha}\otimes \K)$, which satisfies
$$\lim_{n\to\infty} f_nbf_n^*=(a-\varepsilon)_+.$$
Thus, $(a-\varepsilon)_+\precsim b$ in $\widetilde{A^\alpha}\otimes \K$. Since $\varepsilon>0$ is arbitrary, we conclude that $[a]\le [b]$ in $\Cu(\widetilde{A^\alpha})$, as desired.

(ii) Since $A$ is an ideal in $\widetilde A$, the semigroup $\Cu(A)$ can be identified with the subsemigroup of $\Cu(\widetilde A)$ given by
$$\{[a]\in \Cu(\widetilde A)\colon a\in (A\otimes \K)_+\}.$$
Using this identification, it is clear that the restriction of $\Cu(\widetilde i)$ to $\Cu(A)$ is $\Cu(i)$. Therefore, it follows from the first part of the theorem that $\Cu(i)$ is an order embedding.

Let us now proceed to prove the equalities stated in the theorem. It is sufficient to show that
\begin{align}\label{eq: inclusions}
\mathrm{Im}(\Cu(i))\subseteq\overline{\mathrm{Im}\left(\sum\limits_{g\in G}\Cu(\alpha_g)\right)}\subseteq \Cu(A)^{\Cu(\alpha)}\subseteq \Cu(A)^{\Cu(\alpha)}_\N\subseteq \mathrm{Im}(\Cu(i)).
\end{align}
The third inclusion is immediate and true in full generality. The second inclusion follows using that for $[a]\in \Cu(A)$, the element $\sum\limits_{g\in G}\Cu(\alpha_g)([a])$ is $\Cu(\alpha)$-invariant, that $\sum\limits_{g\in G}\Cu(\alpha_g)([a])$ is the supremum of the path
$$t\mapsto \sum\limits_{g\in G}\Cu(\alpha_g)([(a+t-1)_+]),$$
and that $\Cu(A)^{\Cu(\alpha)}=\overline{\Cu(A)^{\Cu(\alpha)}}$ by part (ii) of Lemma \ref{lem: closure}.

We proceed to show the first inclusion. Fix a positive element $a\in A^\alpha\otimes \K$ and let $\varepsilon>0$. Using the Rokhlin property for $\alpha\otimes \id_\K$ with $F=\{a\}$, choose orthogonal positive contractions $(r_g)_{g\in G}\subseteq A\otimes \K$ such that
\begin{align}\label{eq: inequalities}
\left\|a-\sum\limits_{g\in G}r_gar_g\right\|<\varepsilon \ \ \mbox{ and } \ \ \left\|\alpha_g(r_ear_e)-r_gar_g\right\|<\varepsilon,
\end{align}
for all $g\in G$. Using the first inequality above and Lemma \ref{lem: Cuntz relation}, we obtain
\begin{align*}
\left[\left(a-4\varepsilon\right)_+\right]\le \left[\left(\sum\limits_{g\in G}r_gar_g-3\varepsilon\right)_+\right]\le
\left[ \left(\sum\limits_{g\in G}r_gar_g-\varepsilon\right)_+\right]\le [a].
\end{align*}
Furthermore, using the second inequality in \eqref{eq: inequalities} and again using
Lemma \ref{lem: Cuntz relation}, we deduce that
$$\left[\left(r_gar_g-3\varepsilon\right)_+\right]\le \left[\left(\alpha_g(r_ear_e)-2\varepsilon\right)_+\right]\le \left[\left(r_gar_g-\varepsilon\right)_+\right].$$
Take the sum of the previous inequalities, add them over $g\in G$, and use that $\Cu(\alpha_g)[(r_ear_e-2\varepsilon)_+]=[(\alpha_g(r_ear_e)-2\varepsilon)_+]$,
to conclude that
$$\left[\left(a-4\varepsilon\right)_+\right]\ll \sum\limits_{g\in
G}\Cu(\alpha_g)\left[\left(r_ear_e-2\varepsilon\right)_+\right]\le [a].$$

We have shown that for every $\varepsilon>0$, there is an element $x$ in $\mathrm{Im}\left(\sum\limits_{g\in G}\Cu(\alpha_g)\right)$ such that
$$[(a-\varepsilon)_+]\ll x\le [a].$$
By Lemma \ref{lem: sup} applied to $[a]=\sup\limits_{\varepsilon>0} [(a-\varepsilon)_+]$ and to the set $S=\mathrm{Im}\left(\sum\limits_{g\in G}\Cu(\alpha_g)\right)$, it follows that $[a]$ is the supremum of an increasing sequence in $\mathrm{Im}\left(\sum\limits_{g\in G}\Cu(\alpha_g)\right)$, showing that the first inclusion in (\ref{eq: inclusions}) holds.

In order to complete the proof, let us show that the fourth inclusion in \eqref{eq: inclusions} is also true. Fix $x\in \Cu(A)^{\Cu(\alpha)}_\N$. Choose a rapidly increasing sequence $(x_n)_{n\in\N}$ in $\Cu(A)$ such that $\Cu(\alpha_g)(x_n)=x_n$ for all $n\in \N$ and all for all $g\in G$. Fix $m\in \N$ and consider the elements $x_n$ with $n\geq m$. Note that $x_m\ll x_{m+1}\ll \cdots \ll x$. By Lemma \ref{lem: interpolation}, there is a positive element $a\in A\otimes\K$ such that
$$x_m\ll [(a-3\varepsilon)_+]\ll x_{m+1}\ll (a-2\varepsilon)_+\ll x_{m+2}\ll (a-\varepsilon)_+\ll x= [a].$$
Note that this implies that
$$[\alpha_g(a)]=\Cu(\alpha_g)[a]=\Cu(\alpha_g)(x)=x=[a]\le [a]$$
and
$$[(a-2\varepsilon)_+]\le x_{m+2}=\Cu(\alpha_g)(x_{m+2})\le\Cu(\alpha_g)[(a-\varepsilon)_+]=[\alpha_g((a-\varepsilon)_+)]$$
for every $g\in G$. By the definition of Cuntz subequivalence, there are elements $f_g, h_g\in A\otimes \K$ for $g\in G$ such that
$$\|\alpha_g(a)-f_gaf_g^*\|<\frac{\varepsilon}{|G|}$$
and
$$\|(a-2\varepsilon)_+-h_g\alpha_g((a-\varepsilon)_+)h_g^*\|<\frac{\varepsilon}{|G|}.$$
Using the Rokhlin property for $\alpha$, with
$$F=\{\alpha_g(a), \alpha_g((a-\varepsilon)_+), f_g, h_g\colon g\in G\}\cup \{(a-2\varepsilon)_+\},$$
choose positive orthogonal contractions $(r_g)_{g\in G}\subseteq (A\otimes \K)^\infty\cap F'$ as in (ii) of Lemma \ref{lem: Rokhlin equivalence}. Set $f=\sum\limits_{g\in G}f_gr_g$ and $h=\sum\limits_{g\in G}h_gr_g$. Then
\begin{align*}
&\left\|\sum\limits_{g\in G} r_g\alpha_g(a)r_g-faf^*\right\|=\left\|\sum\limits_{g\in G}r_g(\alpha_g(a)-f_gaf_g^*)\right\|<|G|\cdot \frac{\varepsilon}{|G|}=\varepsilon,
\end{align*}
in $(A\otimes \K)^\infty$.
Similarly,
\begin{align*}
\left\|(a-2\varepsilon)_+-h\left(\sum\limits_{g\in G}\alpha_g((a-\varepsilon)_+)\right)h^*\right\|<\varepsilon.
\end{align*}
Using that $r_g$ commutes with $\alpha_g(a)$ and that $r_g^2\alpha_g(a)=r_g\alpha_g(a)$ for all $g\in G$, one easily shows that
$$\sum\limits_{g\in G} r_g\alpha_g(a)r_g=\sum\limits_{g\in G}\alpha(r_ear_e), \quad \text{ and }\quad r_g(\alpha_g((a-\varepsilon)_+))r_g=(r_g\alpha_g(a)r_g-\varepsilon)_+,$$
for all $g\in G$. Thus, we have
\begin{align*}
\sum\limits_{g\in G}r_g(\alpha_g((a-\varepsilon)_+))r_g&=\sum\limits_{g\in G}r_g(\alpha_g(a)-\varepsilon)_+r_g\\
&=\sum\limits_{g\in G}(r_g\alpha_g(a)r_g-\varepsilon)_+\\
&=\left(\sum\limits_{g\in G}r_g\alpha_g(a)r_g-\varepsilon\right)_+\\
&=\left(\sum\limits_{g\in G} \alpha_g(r_ear_e)-\varepsilon\right)_+.
\end{align*}
Therefore, we conclude that
$$\left\|\sum\limits_{g\in G}\alpha_g(r_ear_e)-faf^*\right\|<\varepsilon, \quad \text{and}\quad\left\|(a-2\varepsilon)_+-h\left(\sum\limits_{g\in G} \alpha_g(r_ear_e)-\varepsilon\right)_+h^*\right\|<\varepsilon.$$
Let $(r_n)_{n\in \N}$, $(f_n)_{n\in \N}$, and $(h_n)_{n\in \N}$ be representatives of $r_e$, $f$, and $h$ in $\ell^\infty(\N, A\otimes \K)$, with $r_n$ positive for all $n\in \N$. By the previous inequalities, there exists $k\in \N$ such that
$$\left\|\sum\limits_{g\in G}\alpha_g(r_kar_k)-f_kaf_k^*\right\|<\varepsilon, \ \text{ and } \ \left\|(a-2\varepsilon)_+-h_k\left(\sum\limits_{g\in G} \alpha_g(r_kar_k)-\varepsilon\right)_+h_k^*\right\|<\varepsilon$$
hold in $A\otimes \K$. By Lemma \ref{lem: Cuntz relation} applied to the elements $\sum\limits_{g\in G}\alpha_g(r_kar_k)$ and $f_kaf_k^*$, and to the elements $(a-2\varepsilon)_+$ and $h_k\left(\sum\limits_{g\in G}\alpha(r_kar_k)-\varepsilon\right)_+h_k^*$, we deduce that
$$[(a-3\varepsilon)_+]\le \left[\left(\sum\limits_{g\in G}\alpha_g(r_kar_k)-\varepsilon\right)_+\right]\le [a].$$
Therefore,
$$x_m\ll \left[\left(\sum\limits_{g\in G}\alpha_g(r_kar_k)-\varepsilon\right)_+\right]\ll x.$$

Note that the element $\left(\sum\limits_{g\in G}\alpha_g(r_kar_k)-\varepsilon\right)_+$ belongs to $(A\otimes\K)^{\alpha}$ and so it is in the image of the inclusion map $i^s=i\otimes\id_\K\colon (A\otimes\K)^{\alpha}\to A\otimes \K$. Since $m$ is arbitrary, we deduce that $x$ is the supremum of an increasing sequence in $\mathrm{Im}(\Cu(i))$ by Lemma \ref{lem: sup}. Choose a sequence $(y_n)_{n\in\N}$ in $\Cu(A^\alpha)$ such that $(\Cu(i)(y_n))_{n\in\N}$ is increasing in $\Cu(A)$ and set $x=\sup\limits_{n\in\N} (\Cu(i)(y_n))$. Since $\Cu(i)$ is an order embedding, it follows that $(y_n)_{n\in \N}$ is itself increasing in $\Cu(A^\alpha)$. Set $y=\sup\limits_{n\in\N} y_n$. Then $\Cu(y)=x$ since $\Cu(i)$ preserves suprema of increasing sequences.
\end{proof}

\begin{corollary}\label{cor: Ctz smgp of cp}
Let $A$ be a C*-algebra and let $\alpha$ be an action of a finite group $G$ on $A$ with the Rokhlin property. Then $\Cu(A\rtimes_\alpha G)$ is order-isomorphic to the semigroup:
\begin{align*}
\left\{x\in \Cu(A)\colon \exists \ (x_n)_{n\in\N} \mbox{ in } \Cu(A)\colon
\begin{aligned}
& x_n\ll x_{n+1} \ \forall n\in\N \mbox{ and } x=\sup\limits_{n\in\N} x_n,\\
& \ \Cu(\alpha_g)(x_n)=x_n \ \forall g\in G, \forall n\in\N
\end{aligned}
\right\}.
\end{align*}
\end{corollary}
\begin{proof} Since $\alpha$ has the Rokhlin property, the fixed point algebra $A^\alpha$ is Morita equivalent to the crossed product $A\rtimes_\alpha G$ by \cite[Theorem 2.8]{Phillips-Freeness-of-actions}. Therefore, there is a natural isomorphism $\Cu(A\rtimes_\alpha G)\cong \Cu(A^\alpha)$. Denote by $i\colon A^\alpha\to A$ the natural embedding. By Theorem \ref{thm: Rokhlin Constraint}, the semigroup $\Cu(A^\alpha)$ can be naturally identified with its image under the order embedding $\Cu(i)$, which is $\Cu(A)_\N^{\Cu(\alpha)}$ again by Theorem \ref{thm: Rokhlin Constraint}. The result follows.\end{proof}

\begin{corollary}\label{cor: n-divisible}
Let $A$ be a C*-algebra, let $\alpha$ be an action of a finite group $G$ on $A$ with the Rokhlin property, and set $n=|G|$. Suppose that $\Cu(\alpha_g)=\id_{\Cu(A)}$ for every $g\in G$, and that the map multiplication by $n$ on $\Cu(A)$ is an order embedding (in other words, whenever $x,y\in\Cu(A)$ satisfy $nx\leq ny$, one has $x\leq y$.) Then the map multiplication by $n$ in $\Cu(A)$ is an order-isomorphism.
\end{corollary}
\begin{proof} It suffices to show that for all $x\in\Cu(A)$, there exists $y\in \Cu(A)$ such that $x=ny$. By Theorem \ref{thm: Rokhlin Constraint} (ii), we have
$$\overline{\mathrm{Im}\left(\sum\limits_{g\in G}\Cu(\alpha_g)\right)}=\Cu(A)_\N^{\Cu(\alpha)}.$$
Since $\Cu(\alpha_g)=\id_{\Cu(A)}$ for all $g\in G$, this identity can be rewritten as
$$\overline{n\Cu(A)}=\Cu(A).$$
In particular, if $x$ is an element in $\Cu(A)$, then there exists a sequence $(y_k)_{k\in\N}$ in $\Cu(A)$ such that $(ny_k)_{k\in\N}$ is increasing and $x=\sup\limits_{k\in\N}(ny_k)$. Since $(ny_k)_{k\in\N}$ is increasing, it follows from our assumptions that $(y_k)_{k\in\N}$ is increasing as well. Set $y=\sup\limits_{k\in\N}y_k$. Then
$$x=\sup\limits_{k\in\N} (ny_k)=n\sup\limits_{k\in\N} y_k=ny,$$
and the claim follows.
\end{proof}

Let $A$ be a C*-algebra and let $p$ and $q$ be projections in $A$.
We say that $p$ and $q$ are \emph{Murray-von Neumann equivalent}, and
denote this by $p\sim_{\mathrm{MvN}}q$, if there exists $v\in A$ such that $p=v^*v$ and $q=vv^*$.
We say that $p$ is \emph{Murray-von Neumann subequivalent} to $q$, and denote this by $p\precsim_{\mathrm{MvN}} q$, if there is a projection $p'\in A$ such that $p\sim_{\mathrm{MvN}}p'$ and $p'\le q$.
The projection
$p$ is said to be \emph{finite} if whenever $q$ is a projection in $A$ with $q\leq p$ and $q\sim_{\mathrm{MvN}}p$, then $q=p$.

If $A$ is unital, then $A$ is said to be \emph{finite} if its unit is a finite projection. Moreover, $A$ is said to be
\emph{stably finite} if $\M_n(A)$ is finite for all $n\in \N$. If $A$ is not unital, we say that $A$ is (stably) finite if so
is its unitization $\widetilde{A}$.

\begin{lemma}\label{lem: projections}
Let $A$ be a stably finite C*-algebra and let $p\in A\otimes \K$ be a projection. Suppose that there are positive elements $a,b\in A\otimes\K$ such that $[p]=[a]+[b]$ in $\Cu(A)$. Then $a$ and $b$ are Cuntz equivalent to projections in $A\otimes\K$ (see the comments before Lemma~2.4 for the definition of Cuntz equivalence).
\end{lemma}
\begin{proof}
Let $a$ and $b$ be elements in $A\otimes\K$ as in the statement. By the comments before Lemma \ref{lem: morphism}, we have
$$[a]=\sup\limits_{\varepsilon>0} [(a-\varepsilon)_+] \ \ \mbox{ and } \ \ [b]=\sup\limits_{\varepsilon>0} [(b-\varepsilon)_+].$$
Since $[p]\ll [p]$, there exists $\varepsilon>0$ such that $[p]=[(a-\varepsilon)_+]+[(b-\varepsilon)_+]$. Choose a function
$f_\varepsilon\in \mathrm C_0(0,\infty)$ that is zero on the interval $[\varepsilon, \infty)$, nonzero at every point of $(0,\varepsilon)$
and $\|f_\varepsilon\|_\infty\leq 1$. Then
$$[p]+[f_\varepsilon(a)]+[f_\varepsilon(b)]= [(a-\varepsilon)_+]+[f_\varepsilon(a)]+[(b-\varepsilon)_+]+[f_\varepsilon(b)]\le [a]+[b]=[p].$$
Hence,
$[p]+[f_\varepsilon(a)]+[f_\varepsilon(b)]=[p]$.
Choose $c\in (A\otimes \K)_+$ such that $[c]=[f_\varepsilon(a)]+[f_\varepsilon(b)]$ and $cp=0$. Then $p+c\precsim p$. By
\cite[Lemma 2.3 (iv)]{Kirchberg-Rordam}, for every $\delta>0$ there exists $x\in A\otimes \K$ such that
$$p+(c-\delta)_+=x^*x, \quad xx^*\in p(A\otimes \K)p.$$
Fix $\delta>0$ and let $x$ be as above. Let $x=v|x|$ be the polar decomposition of $x$ in the bidual of $A\otimes \K$. Set $p'=vpv^*$ and
$c'=v(c-\delta)_+v^*$. Then $p'$ is a projection, $p'$ and $c'$ are orthogonal, $p$ and $p'$ are Murray-von Neumann equivalent, and
$p'+c'\in pAp$. Using stable finiteness of $A$ we conclude that $p=p'$ and $c'=0$. It follows that $(c-\delta)_+=0$ for all $\delta>0$,
and thus $c=0$. Hence, $f_\varepsilon(b)=f_\varepsilon(a)=0$ and in particular, $a$ and $b$ have a gap in their spectra. Therefore, they
are Cuntz equivalent to projections.
\end{proof}

Recall that the \emph{Murray-von Neumann semigroup} of $A$, denoted by $\mathrm V(A)$, is defined as the quotient of the set of projections of $A\otimes \K$ by the Murray-von Neumann equivalence relation.

Note that $p\precsim_{\Cu} q$ if and only if $p\precsim_{\mathrm{MvN}} q$.
On the other hand, $p\precsim_{\mathrm{MvN}} q$ and $q\precsim_{\mathrm{MvN}} p$ do not in general imply that $p\sim_{\mathrm{MvN}} q$, although this is the case whenever $A$ is finite.
In particular, if $A$ is finite, then $p\sim_{\Cu} q$ if and only if $p\sim_{\mathrm{MvN}} q$.
Hence, if $A$ is stably finite, then the semigroup $\mathrm{V}(A)$ can be identified with the ordered subsemigroup of $\Cu(A)$ consisting of the Cuntz equivalence classes of projections of $A\otimes \K$.

Recall that if $S$ is a semigroup in $\CCu$ and $x$ and $y$ are elements of $S$, we say that $x$ is \emph{compactly contained} in $y$, and denote this by $x\ll y$, if for every increasing sequence $(y_n)_{n\in\N}$ in $S$ such that $y=\sup\limits_{n\in\N}y_n$, there exists $n_0\in\N$ such that $x\leq y_n$ for all $n\geq n_0$.

\begin{definition}\label{df: compact} Let $S$ be a semigroup in $\CCu$ and let $x$ be an element of $S$. We say that $x$ is \emph{compact} if $x\ll x$. Equivalently, $x$ is compact if whenever $(x_n)_{n\in\N}$ is a sequence in $S$ such that $x=\sup\limits_{n\in\N}x_n$, then there exists $n_0\in\N$ such that $x_n=x$ for all $n\geq n_0$.\end{definition}

It is easy to check that the Cuntz class $[p]\in\Cu(A)$ of any projection $p$ in a C*-algebra $A$ (or in $A\otimes\K$) is a compact element in $\Cu(A)$. Moreover, when $A$ is stably finite, then every compact element of $\Cu(A)$ is the Cuntz class of a projection in $A\otimes \K$ by \cite[Theorem 3.5]{Brown-Ciuperca}. In particular, $\mathrm V(A)$ can be identified with the semigroup of compact elements of $\Cu(A)$ if $A$ is a stably finite C*-algebra.

When studying stably finite C*-algebras in connection with finite group actions with the Rokhlin property, the following lemma is often times useful. The result may be interesting in its own right, and could have been proved in \cite{Osaka-Phillips} since it is a direct application of their methods.

\begin{lemma}\label{lem: stable finiteness preserved}
Let $G$ be a finite group, let $A$ be a unital stably finite C*-algebra and let $\alpha\colon G\to\Aut(A)$ be an action with the Rokhlin property. Then the crossed product $A\rtimes_\alpha G$ and the fixed point algebra $A^\alpha$ are stably finite.\end{lemma}
\begin{proof} The fixed point algebra $A^\alpha$, being a unital subalgebra of $A$, is stably finite. On the other hand, the crossed product $A\rtimes_\alpha G$, being stably isomorphic to $A^\alpha$ by \cite[Theorem 2.8]{Phillips-Freeness-of-actions}, must itself also be stably finite.
\end{proof}


For unital, simple C*-algebras, part (ii) of the theorem below was first proved by Izumi in \cite{Izumi-I}.
The proof in our context follows completely different ideas.

\begin{theorem}\label{thm: image of inclusion on K-theory}
Let $A$ be a stably finite C*-algebra and let $\alpha$ be an action of a finite group $G$ on $A$ with the Rokhlin property. Let $i\colon A^\alpha\to A$ be the inclusion map.
\begin{itemize}
\item[(i)] The map $\mathrm V(i)\colon \mathrm V(A^\alpha)\to \mathrm V(A)$ is an order embedding and
\begin{align*}
&\mathrm{Im}(\mathrm V(i))=\mathrm{Im}\left(\sum\limits_{g\in G}\mathrm V(\alpha_g)\right)=\left\{x\in \mathrm V(A)\colon \mathrm V(\alpha_g)(x)=x, \, \forall  g\in G\right\}.
\end{align*}
\item[(ii)] If $A$ has an approximate identity consisting of projections, then $\mathrm K_0(i)\colon \mathrm K_0(A^\alpha)\to \mathrm K_0(A)$ is an order embedding and
\begin{align*}
\mathrm{Im}(\mathrm K_0(i))=\mathrm{Im}\left(\sum\limits_{g\in G}\mathrm K_0(\alpha_g)\right)=\left\{x\in \mathrm K_0(A)\colon \mathrm K_0(\alpha_g)(x)=x, \, \forall  g \in G\right\}.
\end{align*}
\end{itemize}
\end{theorem}
\begin{proof}
(i) The fact that $\mathrm V(i)$ is an order embedding is a consequence of Theorem \ref{thm: Rokhlin Constraint} and the remarks before and after Definition~\ref{df: compact}. Let us now show the inclusions
\begin{equation}\label{inclusions} \mathrm{Im}(\mathrm V(i))\subseteq\mathrm{Im}\left(\sum\limits_{g\in G}\mathrm V(\alpha_g)\right)\subseteq\left\{x\in \mathrm V(A)\colon \mathrm V(\alpha_g)(x)=x \ \forall \ g\in G\right\}\subseteq \mathrm{Im}(\mathrm V(i)).\end{equation}
Let $p\in A^\alpha\otimes \K$ be a projection. By Theorem \ref{thm: Rokhlin Constraint}, there exists a sequence $(a_n)_{n\in\N}$ in $(A\otimes\K)_+$ such that $\left(\sum\limits_{g\in G}\Cu(\alpha_g)([a_n])\right)_{n\in\N}$ is increasing and
$$[i(p)]=\sup\limits_{n\in\N}\left(\sum\limits_{g\in G}\Cu(\alpha_g)([a_n])\right).$$
Since $[i(p)]$ is a compact element in $\Cu(A)$, it follows that there exists $n_0\in \N$ such that $[i(p)]=\sum\limits_{g\in G}\Cu(\alpha_g)([a_n])$ for all $n\geq n_0$. Fix $m\geq n_0$. It is easy to check that if $S$ is a semigroup in the category $\CCu$, then a sum of elements in $S$ is compact if and only if each summand is compact. It follows that $\Cu(\alpha_g)([a_m])$ is compact for all $g\in G$. In particular, and denoting the unit
of $G$ by $e$, we deduce that $[a_m]=\Cu(\alpha_e)([a_m])$ is compact. Since $A$ is stably finite by assumption, there exists a projection $q\in A\otimes\K$ such that $[q]=[a_m]$. Thus
$$\mathrm V(i)([p])=\sum\limits_{g\in G}\mathrm V(\alpha_g)([q])\in \mathrm{Im}\left(\sum\limits_{g\in G}\mathrm V(\alpha_g)\right),$$
showing that the first inclusion in (\ref{inclusions}) holds.

Using the fact that $\alpha_h\circ \left(\sum\limits_{g\in G}\alpha_g\right)=\sum\limits_{g\in G}\alpha_g$ for all $h\in G$, it is easy to check that
$$\mathrm{Im}\left(\sum\limits_{g\in G}\mathrm V(\alpha_g)\right)\subseteq\left\{x\in \mathrm V(A)\colon \mathrm V(\alpha_g)(x)=x, \, \forall  g\in G\right\},$$
thus showing that the second inclusion also holds.

We proceed to prove the third inclusion. Let $x\in \mathrm{V}(A)$ be such that $\mathrm{V}(\alpha_g)(x)=x$ for all $g\in G$. Note that $x$ is compact as an element in $\Cu(A)$. It follows that $\Cu(\alpha_g)(x)=x$ for all $g\in G$ and hence by Theorem \ref{thm: Rokhlin Constraint} there exists $a\in (A^\alpha\otimes\K)_+$ such that $\Cu(i)([a])=x$. Since the map $\Cu(i)$ is an order embedding again by Theorem \ref{thm: Rokhlin Constraint}, one concludes that $[a]$ is compact.

Finally, the fixed point algebra $A^\alpha$ is stably finite by Lemma \ref{lem: stable finiteness preserved} and thus there is a projection $p\in A^\alpha\otimes\K$ such that $[p]=[a]$ in $\Cu(A^\alpha)$. It follows that $\Cu(i)([p])=x$, showing that the third inclusion in (\ref{inclusions}) is also true.

(ii) Follows using the first part, together with the fact that the $\mathrm K_0$-group of a C*-algebra containing an approximate identity consisting of projections, agrees with the Grothendieck group of the Murray-von Neumann semigroup of the algebra; see Proposition~5.5.5 in \cite{Blackadar-book}.
\end{proof}

In the following corollary, the picture of $\mathrm V(A\rtimes_\alpha G)$ is valid for arbitrary $A$.

\begin{corollary}\label{cor: K-thy of cp} Let $A$ be a stably finite C*-algebra containing an approximate identity
consisting of projections, and let $\alpha$ be an action of a finite group $G$ on $A$ with the Rokhlin property. Then there are isomorphisms
\begin{align*}
\mathrm V(A\rtimes_\alpha G)&\cong \left\{x\in \mathrm V(A)\colon \mathrm V(\alpha_g)(x)=x, \, \forall g\in G\right\},\\
\mathrm K_\ast(A\rtimes_\alpha G)&\cong \left\{x\in \mathrm K_\ast(A)\colon \mathrm K_\ast(\alpha_g)(x)=x, \, \forall g\in G\right\}.
\end{align*}\end{corollary}
\begin{proof} Recall that if $\alpha$ has the Rokhlin property, then the fixed point algebra $A^\alpha$ and the crossed product $A\rtimes_\alpha G$ are Morita equivalent, and hence have isomorphic $\mathrm K$-theory and Murray-von Neumann semigroup. The isomorphisms for $\mathrm V(A\rtimes_\alpha G)$ and $\mathrm K_0(A\rtimes_\alpha G)$ then follow from Theorem \ref{thm: image of inclusion on K-theory} above.

Denote $B=A\otimes C(S^1)$ and give $B$ the diagonal action $\beta=\alpha\otimes\id_{C(S^1)}$ of $G$. Note that $B$ is stably finite and has an approximate identity
consisting of projections, and that $\beta$ has the Rokhlin property by part (i) of Proposition \ref{Rp properties}. Moreover, there is a natural isomorphism $B\rtimes_\beta G\cong (A\rtimes_\alpha G) \otimes C(S^1)$. Applying the K\"unneth formula in the first step, together with the conclusion of this proposition for $\mathrm K_0$ (which was shown to hold in the paragraph above) in the second step, and again the K\"unneth formula in the fourth step, we obtain
\begin{align*}
\left\{x\in \mathrm K_\ast(A)\colon \mathrm K_\ast(\alpha_g)(x)=x, \ \forall \ g\in G\right\}&\cong \left\{x\in \mathrm K_0(B)\colon \mathrm K_0(\beta_g)(x)=x, \ \forall \ g\in G\right\}\\
&\cong \mathrm K_0(B\rtimes_\beta G)\\
&\cong \mathrm K_0((A\rtimes_\alpha G)\otimes C(S^1))\\
&\cong \mathrm K_\ast(A\rtimes_\alpha G),\end{align*}
as desired.
\end{proof}

\section{Equivariant UHF-absorption}
In this section, we study absorption of UHF-algebras in relation to the Rokhlin property. We show that for a certain class of C*-algebras, absorption of a UHF-algebra of infinite type is equivalent to existence of an action with the Rokhlin property that is pointwise approximately inner. (The cardinality of the group is related to the type of the UHF-algebra.) Moreover, in this case, not only the C*-algebra absorbs the corresponding UHF-algebra, but also the action in question absorbs the model action constructed in Example \ref{eg: model action}. Thus, Rokhlin actions allow us to prove that certain algebras are \emph{equivariantly} UHF-absorbing.

\subsection{Unique $n$-divisibility.}

The goal of this section is to show that for certain C*-algebras, absorption of the UHF-algebra of type $n^\infty$ is equivalent to its Cuntz semigroup being $n$-divisible. Along the way, we show that for a C*-algebra $A$, the Cuntz semigroups of $A$ and of $A\otimes \M_{n^\infty}$ are isomorphic if and only if $\Cu(A)$ is uniquely $n$-divisible.

We point out that some of the results of this section, particularly Theorem~\ref{Ctz smgp and unique n divis},
were independently obtained in the recent preprint \cite{APT}, as applications of their theory of tensor
products of Cuntz semigroups. On the other hand, the proofs we give here are direct and elementary. 
Additionally, our techniques also apply to other functors, for instance the functor $\Cu^\sim$.

We begin defining the main notion of this section. Recall that if $S$ and $T$ are ordered semigroup and $\varphi\colon S\to T$ is a semigroup homomorphism, we say that $\varphi$ is an \emph{order embedding} if $\varphi(s)\leq \varphi(s')$ implies $s\leq s'$ for all $s, s'\in S$. A semigroup isomorphism is called an \emph{order preserving} semigroup isomorphism if it is an order embedding.

\begin{definition}\label{df: uniquely n div}
Let $S$ be an ordered semigroup and let $n$ be a positive integer.
\begin{enumerate}
\item We say that $S$ is \emph{$n$-divisible}, if for every $x$ in $S$ there exists $y$ in $S$ such that $x=ny$.
\item We say that $G$ is \emph{uniquely $n$-divisible}, if multiplication by $n$ on $S$ is an order preserving semigroup isomorphism.
\end{enumerate}
\end{definition}

Recall that the category $\CCu$ is closed under sequential inductive limits.

\begin{lemma}\label{lem: n-div-inductivelim}
Let $n\in \N$ and let $S$ be a semigroup in the category $\CCu$. Denote by $\rho\colon S\to S$ the map given by $\rho(s)=ns$ for
all $s\in S$. Let $T$ be the semigroup in $\CCu$ obtained as the inductive limit
of the sequence
\begin{align*}
\xymatrix{
S\ar[r]^-{\rho} &S\ar[r]^-{\rho} & S\ar[r]^-{\rho} & \cdots.
}
\end{align*}
Then $T$ is uniquely $n$-divisible.
\end{lemma}
\begin{proof}
Let $S$ and $T$ be as in the statement. To avoid any confusion with the notation, we will denote the map between the $k$-th and
$(k+1)$-st copies of $S$ by $\rho_k$, so we write $T$ as the direct limit
\begin{align*}
\xymatrix{
S\ar[r]^-{\rho_1} &S\ar[r]^-{\rho_2} & S\ar[r]^-{\rho_3} &\ar[r] \cdots & T.
}
\end{align*}
For $k,m\in\N$ with $m>k$, we let $\rho_{k,m}\colon S\to S$ denote the composition $\rho_{m-1}\circ\rho_{m-2}\circ\cdots\circ\rho_k$, and we let
$\rho_{k,\infty}\colon S\to T$ denote the canonical map from the $k$-th copy of $S$ to $T$.

Let $s, t\in T$ satisfy $ns\le nt$. By part (i) of Proposition \ref{prop: inductivelimitCu}, there exist sequences $(s_k)_{k\in\N}$ and
$(t_k)_{k\in \N}$ in $S$ such that
\begin{align*}
\rho_k(s_k)\ll s_{k+1} \mbox{ for all } k\in\N \ \ &\mbox{ and } \ \ s=\sup\limits_{k\in\N} \rho_{k,\infty}(s_k)\\
\rho_k(t_k)\ll t_{k+1} \mbox{ for all } k\in\N \ \ &\mbox{ and } \ \ t=\sup\limits_{k\in\N} \rho_{k,\infty}(t_k).
\end{align*}
It follows that $\rho_{k,\infty}(s_k)\ll \rho_{k+1,\infty}(s_{k+1})$ and $\rho_{k,\infty}(t_k)\ll \rho_{k+1,\infty}(t_{k+1})$ for all $k\in \N$.

Let $k\geq 2$ be fixed. Since
$$\rho_{k,\infty}(n s_k)\ll n s\le n t=\sup\limits_{k\in\N} \rho_{k,\infty}(n t_k),$$
there exists $l\in \N$ such that $\rho_{k,\infty}(n s_k)\le \rho_{l,\infty}(n t_l)$. Use part (ii) of Proposition
\ref{prop: inductivelimitCu} and $\rho_{j-1}(s_{j-1})\ll s_j$ for all $j\in \N$, to choose $m\ge k,l$ such that
$\rho_{k-1,m}(n s_{k-1})\le \rho_{l,m}(n t_l)$. Therefore,
\begin{align*}
\rho_{k-1,\infty}(s_{k-1})&=\rho_{m+1,\infty}(\rho_{m+1}(\rho_{k-1,m}(s_{k-1})))\\
&=\rho_{m+1,\infty}(n\rho_{k-1,m}(s_{k-1}))\\
&=\rho_{m+1,\infty}(\rho_{k-1,m}(ns_{k-1}))\\
&\le \rho_{m+1,\infty}(\rho_{l,m}(n t_l))\\
&=\rho_{m+1,\infty}(n\rho_{l,m}(t_l))\\
&=\rho_{m+1,\infty}(\rho_{m+1}(\rho_{l,m}(t_l)))\\
&=\rho_{l,\infty}(t_l)\\
&\le t,
\end{align*}
this is, $\rho_{k-1,\infty}(s_{k-1})\le t$. Since this holds for all $k\ge 2$, we conclude that
$$s=\sup\limits_{k\geq 2} \rho_{k-1,\infty}(s_{k-1})\le t.$$
We have shown that $ns\le nt$ in $T$ implies $s\le t$. In other words, multiplication by $n$ on $T$ is an order embedding, as desired.

To conclude the proof, let us show that $T$ is $n$-divisible. Fix $t\in T$ and choose a sequence $(t_k)_{k\in\N}$ in
$T$ satisfying
$$\rho_k(t_k)\ll t_{k+1} \mbox{ for all } k\in\N \ \ \mbox{ and } \ \ t=\sup\limits_{k\in\N} \rho_{k,\infty}(t_k).$$
For each $k\in\N$ we have
\[
\rho_{k,k+2}(t_k)=n^2t_k=n\rho_{k+1,k+2}(t_k).
\]
With $x_k=\rho_{k+1,\infty}(t_k)$, it follows that $\rho_{k,\infty}(t_k)= nx_k$. Since $(\rho_{k,\infty}(t_k))_{k\in\N}$ is an increasing
sequence in $T$, we deduce that $(nx_k)_{k\in\N}$ is an increasing sequence in $T$ as well. Since we have shown in the first part of this
proof that multiplication by $n$ on $T$ is an order
embedding, we conclude that $(x_k)_{k\in\N}$ is also increasing. With $x$ denoting the supremum of $(x_k)_{k\in\N}$, we have
\[
t=\sup\limits_{k\in\N} \rho_{k,\infty}(t_k)=\sup nx_k=n\sup\limits_{k\in\N} x_k=nx,
\]
which completes the proof.
\end{proof}


We point out that the functor $\Cu^\sim$ does not distinguish between *-homomorphisms that are approximately unitarily equivalent (with unitaries taken in the unitization). On the other hand, the corresponding statement for approximate unitary equivalence with unitaries taken in the multiplier algebra is not known in general.
The following proposition, of independent interest, implies that this is the case whenever the codomain has stable rank one. This will be used in the proof of Lemma~\ref{lem: n} to
deduce that certain *-homomorphisms are trivial at the level of $\Cu^\sim$.

\begin{proposition}\label{prop: multiplierau}
Let $A$ and $B$ be C*-algebras with $B$ stable, and let $\phi, \psi\colon A\to B$ be *-homomorphisms. Suppose that $\phi$ and $\psi$ are
approximately unitarily equivalent with unitaries taken in the multiplier algebra of $B$. Then $\phi$ and $\psi$ are approximately unitarily
equivalent with unitaries taken in the unitization of $B$.
\end{proposition}
\begin{proof}
Denote by $\iota\colon B\to \M(B)^\infty$ the canonical inclusion as constant sequences. We will identify $B$ with a
subalgebra of $\M(B)$, and suppress $\iota$ from the notation. Hence we will denote the maps
$\iota\circ\phi,\iota\circ\psi\colon A\to \M(B)^\infty$ again by $\phi$ and $\psi$, respectively.

Let $F\subseteq A$ be a finite set. Then there exists a unitary $u=\pi_{\M(B)}((u_n)_{n\in\N})$ in $\mathrm{M}(B)^\infty$ such that
$\phi(a)=u\psi(a)u^*$ for all $a\in F$. Choose a sequence $(s_n)_{n\in \N}$ of positive contractions in $B$ such that
$$\lim\limits_{n\to \infty} s_n\psi(a)=\lim\limits_{n\to \infty} \psi(a)s_n=\psi(a)$$
for all $a\in F$. Let $s=\pi_{\mathrm M(B)}((s_n)_{n\in\N})$ denote the image of $(s_n)_{n\in\N}$ in $B^\infty\subseteq \mathrm M(B)^\infty $. Then
$$s\phi(a)=\phi(a)s=\phi(a)$$
for all $a\in F$. Since $B$ is stable, we have $B\subseteq \overline{\mathrm{GL}(\widetilde B)}$ by \cite[Lemma 4.3.2]{B-R-T-T-W}. Hence, elements in $B$ have
approximate polar decompositions with unitaries taken in $\widetilde B$. This implies that there exists a sequence $(v_n)_{n\in \N}$
of unitaries in $\widetilde B$ such that $\lim\limits_{n\to \infty}\|u_ns_n-v_ns_n\|=0$. Let $v=\pi_{\widetilde{B}}((v_n)_{n\in \N})$ denote the image
of $(v_n)_{n\in \N}$ in $(\widetilde{B})^\infty$. Then $us=vs$ and
$$\phi(a)=u\psi(a)u^*=us\psi(a)su^*=vs\psi(a)sv^*=v\psi(a)v^*$$
for all $a\in F$. This implies that $\lim\limits_{n\to \infty}\|\phi(a)-v_n\psi(a)v_n^*\|=0$ for all $a\in F$.
Since $v_n$ is a unitary in $\widetilde{B}$ for all $n\in\N$, we conclude that $\phi$ and $\psi$ are approximately unitarily equivalent
with unitaries taken in the unitization of $B$.
\end{proof}

Let $n,k\in \N$. We let $\left(f_{i,j}^{(n^k)}\right)_{i,j=0}^{n^k-1}$ denote the set of matrix units of $\M_{n^k}(\C)$. Recall that
if $A$ and $B$ are C*-algebras and $\phi,\psi\colon A\to B$ are *-homomorphisms with orthogonal ranges, then $\phi+\psi$ is also
a *-homomorphism and $\Cu(\phi+\psi)=\Cu(\phi)+\Cu(\psi)$.
\begin{lemma}\label{lem: n}
Let $A$ be a C*-algebra and let $n, k\in \N$. Let $\iota_k\colon A\to \M_{n^k}(A)$ be the map given by $\iota_k(a)=a\otimes f^{(n^k)}_{0,0}$ for all $a\in A$, and let $j_k\colon \M_{n^k}(A)\to \M_{n^{k+1}}(A)$ be the map given by $j_k(a)=a\otimes 1_n$ for all $a\in \M_{n^k}(A)$.
Then the map
\begin{align*}
\Cu(\iota_{k+1})^{-1}\circ \Cu(j_k)\circ \Cu(\iota_k)\colon \Cu(A)\to \Cu(A),
\end{align*}
is the map multiplication by $n$.
\end{lemma}
\begin{proof}
Since $\Cu $ is invariant under stabilization, we may assume that the algebra $A$ is stable.

Fix $k$ in $\N$. For each $0\le i\le n-1$, let $j_{k,i}\colon \M_{n^k}(A)\to \M_{n^{k+1}(A)}$ be the map defined by
$j_{k,i}(b)=b\otimes f_{i,i}^{(n)}$
for all $b\in \M_{n^k}(A)$. Then the maps $(j_{k,i})_{i=0}^{n-1}$ have orthogonal ranges and $j_k=\sum\limits_{i=1}^{n-1} j_{k,i}$. By the comments
before this lemma, we have
$$\Cu(j_k)=\sum\limits_{i=0}^{n-1}\Cu(j_{k,i}).$$
Since $f_{i,i}^{(n)}$ and $f_{\ell,\ell}^{(n)}$ are unitarily equivalent in $\M_n(\C)$ for all $i,\ell=0,\ldots,n-1$, we conclude that the
maps $j_{k,i}$ and $j_{k,\ell}$ are unitarily equivalent with unitaries in the multiplier algebra of $\M_{n^{k+1}}(A)$. By Proposition
\ref{prop: multiplierau}, this implies that the maps $j_{k,i}$ and $j_{k,\ell}$ are approximately unitarily equivalent (with unitaries taken in the unitization of $\mathrm M_{n^{k+1}}(A)$). Since approximate
unitary equivalent maps yield the same morphism at the level of the Cuntz semigroup, we deduce that $\Cu(j_{k,i})=\Cu(j_{k,\ell})$ for all
$i,\ell=0,\ldots,n-1$.
Given a positive element $a$ in $A\otimes\K$,we have
\begin{align*}
(\Cu(\iota_{k+1})^{-1}\circ\Cu(j_k)\circ \Cu(\iota_k))([a])&=(\Cu(\iota_{k+1})^{-1}\circ\Cu(j_k))\left(\left[a\otimes f_{0,0}^{(n^k)}\right]\right)\\
&=\Cu(\iota_{k+1})^{-1}\left(\sum\limits_{i=0}^{n-1}\Cu(j_{k,i})\left(\left[a\otimes f_{0,0}^{(n^k)}\right]\right)\right)\\
&=\Cu(\iota_{k+1})^{-1}\left(n\Cu(j_{k,0})\left(\left[a\otimes f_{0,0}^{(n^k)}\right]\right)\right)\\
&=n\Cu(\iota_{k+1})^{-1}\left(\left[a\otimes f_{0,0}^{(n^k)}\otimes f_{0,0}^{(n)}\right]\right)\\
&=n\Cu(\iota_{k+1})^{-1}\left(\left[a\otimes f_{0,0}^{(n^{k+1})}\right]\right)\\
&=n [a]
\end{align*}
We conclude that $\Cu(\iota_{k+1})^{-1}\circ \Cu(j_k)\circ \Cu(\iota_k)$ is the map multiplication by $n$.
\end{proof}

\begin{theorem}\label{Ctz smgp and unique n divis}
Let $A$ be a C*-algebra and let $n\in\N$ with $n\ge 2$. Then $\Cu(A)$ is uniquely $n$-divisible if and only if $\Cu(A)\cong \Cu(A\otimes \M_{n^\infty})$ as order semigroups.
\end{theorem}
\begin{proof}
Assume that there exists an isomorphism $\Cu(A)\cong \Cu(A\otimes \M_{n^\infty})$ as ordered semigroups. Using
the inductive limit decomposition $\M_{n^\infty}=\varinjlim \M_{n^k}$ with connecting maps
$j_k\colon \M_{n^{k-1}}(A)\to \M_{n^{k}}(A)$ is given by
$j_k(a)=a\otimes 1_{n}$ for all $a\in \M_{n^{k-1}}(A)$, we can write $A\otimes \M_{n^\infty}$ as the inductive limit
\[
\xymatrix{
A\ar[r]^-{j_1} &\M_n(A)\ar[r]^-{j_2} & \M_{n^2}(A)\ar[r]^-{j_3} &\ar[r] \cdots & A\otimes\M_{n^\infty}.
}
\]
By continuity of the functor $\Cu$ (see \cite[Theorem 2]{Coward-Elliott-Ivanescu}),
the semigroup $\Cu(A\otimes \M_{n^\infty})$ is isomorphic to the inductive limit in the category $\CCu$ of the sequence
\begin{align}\label{A_UHF}
\xymatrix{
\Cu(A)\ar[r]^-{\Cu(j_0)} &\Cu(\M_n(A))\ar[r]^-{\Cu(j_1)} & \Cu(\M_{n^2}(A))\ar[r]^-{\Cu(j_2)} & \cdots.
}
\end{align}
By \cite[Appendix]{Coward-Elliott-Ivanescu}, the inclusion $i_{k}\colon A\to \M_{n^k}(A)$ from $A$ into the upper left corner
of $\M_{n^k}(A)$ induces an isomorphism between the Cuntz semigroup of $A$ and that of $\M_{n^k}(A)$. For $k\in \N$,
let $\varphi_k\colon \Cu(A)\to\Cu(A)$ be given by
$$\varphi_k=\Cu(i_{k+1})^{-1}\circ\Cu(j_{k})\circ\Cu(i_k).$$
The sequence \eqref{A_UHF} implies that $\Cu(A\otimes \M_{n^\infty})$ is the inductive limit of the sequence
\begin{align}\label{A_UHF_2}
\xymatrix{
\Cu(A)\ar[r]^-{\varphi_1} &\Cu(A)\ar[r]^-{\varphi_2} & \Cu(A)\ar[r]^-{\varphi_3} & \cdots.
}
\end{align}
By Lemma \ref{lem: n}, each $\varphi_k$ is the map multiplication by $n$. It follows from Lemma \ref{lem: n-div-inductivelim}
that $\Cu(A\otimes \M_{n^\infty})$ is uniquely $n$-divisible. This shows the ``if" implication.

Conversely, assume that $\Cu(A)$ is uniquely $n$-divisible and adopt the notation used above.
The map $\varphi_k$ is the map
multiplication by $n$ on $\Cu(A)$ by Lemma \ref{lem: n}, so it is an order-isomorphism by assumption.
By the inductive limit expression of $\Cu(A\otimes \M_{n^\infty})$ in \eqref{A_UHF_2}, we conclude that
$\Cu(A)\cong \Cu(A\otimes \M_{n^\infty})$, as desired.
\end{proof}

\begin{remark}
Let $\mathcal{Q}$ denote the universal UHF-algebra. Using the same ideas as in the proof of the previous theorem, one can show that $\Cu(A)\cong\Cu(A\otimes \mathcal{Q})$ if and only if $\Cu(A)$ is uniquely $p$-divisible for every prime number $p$.
\end{remark}

We now turn to direct limits of one-dimensional NCCW-complexes. The following lemma will allow us to reduce to the case where the
algebra itself is a one-dimensional NCCW-complex when proving that multiplication by $n$ is an order embedding at the level of the Cuntz semigroup.

\begin{lemma}\label{lem: cancellation}
Let $(S_k,\rho_k)_{k\in\N}$ be an inductive system in the category $\CCu$, and let $S=\varinjlim (S_k,\rho_k)$ be its inductive
limit in $\CCu$. Let $n\in \N$. If multiplication by $n$ on $S_k$ is an order embedding for all $k$ in $\N$, then the same holds for $S$.
\end{lemma}
\begin{proof}
For $l\geq k$, denote by $\rho_{k,l}\colon S_k\to S_{l+1}$ the composition $\rho_{k,l}=\rho_l\circ\cdots\circ\rho_k$, and denote by
$\rho_{k,\infty}\colon S_k\to S$ the canonical map as in the definition of the inductive limit. Let $s,t\in S$ satisfy $ns\le nt$.
By part (i) of Proposition \ref{prop: inductivelimitCu}, for each $k\in\N$ there exist $s_k, t_k\in S_k$ such that
\begin{align*}
\rho_k(s_k)\ll s_{k+1} \ \ &\mbox{ and } \ \ s=\sup\limits_{k\in\N} \rho_{k,\infty}(s_k),\\
\rho_k(t_k)\ll t_{k+1} \ \ &\mbox{ and } \ \ t=\sup\limits_{k\in\N} \rho_{k,\infty}(t_k).
\end{align*}
Note in particular that $\rho_{k,\infty}(s_k)\ll\rho_{k+1,\infty}(s_{k+1})$ and $\rho_{k,\infty}(t_k)\ll\rho_{k+1,\infty}(t_{k+1})$ for all $k\in \N$.

Fix $k\in \N$. Then
$$\rho_{k,\infty}(ns_k)\ll\rho_{k+1,\infty}(ns_{k+1})\ll \sup\limits_{j\in\N}\rho_{j, \infty}(nt_j).$$
By the definition of the compact containment relation, there exists $j\in\N$ such that
$$\rho_{k,\infty}(ns_k)\ll\rho_{k+1,\infty}(ns_{k+1})\le\rho_{j, \infty}(nt_j).$$
By part (ii) of Proposition \ref{prop: inductivelimitCu}, there exists $l\in\N$ such that
$$n\rho_{k,l}(s_k)=\rho_{k,l}(ns_k)\le\rho_{j, l}(nt_j)=n\rho_{j, l}(t_j).$$
Using that multiplication by $n$ on $S_k$ is an order embedding, we obtain $\rho_{k,l}(s_k)\le \rho_{j, l}(t_j)$. In particular,
$$\rho_{k,\infty}(s_k)\le \rho_{j, \infty}(t_j)\le t.$$
Since $k\in\N$ is arbitrary and $s=\sup\limits_{k\in\N} \rho_{k,\infty}(s_k)$, we conclude that $s\le t$.
\end{proof}

\begin{proposition}\label{pro: NCCW complexes}
Let $A$ be a C*-algebra that can be written as the inductive limit of 1-dimensional NCCW-complexes. Then the endomorphism of $\Cu(A)$ given by multiplication by $n$ is an order embedding.
\end{proposition}
\begin{proof}
By Lemma \ref{lem: cancellation}, it is sufficient to show that the proposition holds when $A$ is a 1-dimensional NCCW-complex. Let $E=\bigoplus_{j=1}^r\M_{k_j}(\C)$ and $F=\bigoplus_{j=1}^s\M_{l_j}(\C)$ be finite dimensional C*-algebras, and for $x\in [0,1]$ denote by $\mathrm{ev}_x\colon \mathrm{C}([0,1], F)\to F$ the evaluation map at the point $x$. Assume that $A$ is given by the pullback decomposition
\begin{equation*}
\xymatrix{A \ar[d] \ar[rr] & & E\ar[d] \\ \mathrm C([0,1], F)\ar[rr]_-{\mathrm{ev}_0\oplus \mathrm{ev}_1} & & F\oplus F,}
\end{equation*}
By \cite[Example 4.2]{Antoine-Perera-Santiago}, the Cuntz semigroup of $A$ is order-isomorphic to a subsemigroup of
$$\mathrm{Lsc}\left([0,1], \overline{\mathbb Z_+}^s\right)\oplus (\overline{\mathbb Z_+})^r.$$
Since multiplication by $n$ on this semigroup is an order embedding, the same holds for any subsemigroup;
in particular, it hold for $\Cu(A)$.
\end{proof}

\begin{corollary}\label{cor: Minfty absorbing}
Let $A$ be a C*-algebra in one of the following classes: unital algebras that can written as inductive limits 1-dimensional NCCW-complexes with trivial $\mathrm K_1$-groups; simple algebras with trivial $\mathrm K_0$-groups that can be written as inductive limits 1-dimensional NCCW-complexes with trivial $\mathrm K_1$-groups; and algebras that can written as inductive limits of punctured-tree algebras. Let $n\in \N$.
Suppose that the map multiplication by $n$ on $\Cu(A)$ is an order-isomorphism. Then $A\cong A\otimes \M_{n^\infty}$.
\end{corollary}
\begin{proof}
By Proposition \ref{pro: NCCW complexes} together with the assumptions in the statement, it follows that the map multiplication by $n$ on $\Cu(A)$ is an order-isomorphism. Hence, it is an isomorphism in the category $\Cu$. By part (ii) of Theorem \ref{Ctz smgp and unique n divis}, there is an isomorphism $\Cu(A)\cong \Cu(A\otimes \M_{n^\infty})$ in $\CCu$.

The same arguments used at the end of the proof of Theorem \ref{thm: mainclassification} show that the classes of C*-algebras in the statement can be classified up to stable isomorphism by their Cuntz semigroup. Therefore, we deduce that
$$A\otimes \K\cong A\otimes \M_{n^\infty}\otimes \K.$$
Using that $\M_{n^\infty}$-absorption is inherited by hereditary C*-subalgebras (\cite[Corollary 3.1]{Toms-Winter}), we conclude that $A\cong A\otimes \M_{n^\infty}$.
\end{proof}

\subsection{Absorption of the model action}

We now proceed to obtain an equivariant UHF-absorption result (compare with \cite[Theorems 3.4 and 3.5]{Izumi-II}).

\begin{theorem}\label{thm: Rp action absorbs model action}
Let $G$ be a finite group and let $A$ be a C*-algebra belonging to one of the classes of C*-algebras described in Corollary \ref{cor: Minfty absorbing}. Then the following statements are equivalent:
\begin{enumerate} \item The C*-algebra $A$ absorbs the UHF-algebra $\M_{|G|^\infty}$.
\item There is an action $\alpha\colon G\to\Aut(A)$ with the Rokhlin property such that $\Cu(\alpha_g)=\id_{\Cu(A)}$ for all $g\in G$.
\item There are actions of $G$ on $A$ with the Rokhlin property, and for any action $\beta\colon G\to\Aut(A)$ with the Rokhlin property and for any action $\delta\colon G\to\Aut(A)$ such that $\Cu(\beta_g)=\Cu(\delta_g)$ for all $g\in G$, one has
    $$(A,\beta)\cong (A\otimes \M_{|G|^\infty},\delta\otimes\mu^G),$$
    that is, there is an isomorphism $\varphi\colon A\to A\otimes \M_{|G|^\infty}$ such that
$$\varphi\circ\beta_g=(\delta\otimes\mu^G)_g\circ\varphi$$
for all $g$ in $G$.\end{enumerate}

In particular, if the above statements hold for $A$, and if $\alpha\colon G\to\Aut(A)$ is an action with the Rokhlin property such that $\Cu(\alpha_g)=\id_{\Cu(A)}$ for all $g\in G$, then $(A,\alpha)\cong (A\otimes \M_{|G|^\infty},\id_A\otimes\mu^G)$.\end{theorem}
\begin{proof}
(i) implies (ii). Fix an isomorphism $\varphi\colon A\to A\otimes \M_{|G|^{\infty}}$ and define an action $\alpha\colon G\to\Aut(A)$ by $\alpha_g=\varphi^{-1}\circ(\id_A\otimes\mu^G)_g\circ\varphi$ for all $g$ in $G$. For a fixed group element $g$ in $G$, the automorphism $\id_A\otimes\mu^G_g$ of $A\otimes\M_{|G|^\infty}$ is approximately inner, and hence so is $\alpha_g$. It follows that $\Cu(\alpha_g)=\id_{\Cu(A)}$ for all $g$ in $G$, as desired.

(ii) implies (i). Assume that there is an action $\alpha\colon G\to\Aut(A)$ with the Rokhlin property such that $\Cu(\alpha_g)=\id_{\Cu(A)}$ for all $g\in G$. Then $A\cong A\otimes \M_{|G|^\infty}$ by Proposition \ref{pro: NCCW complexes}, Corollary \ref{cor: Minfty absorbing} and Corollary \ref{cor: n-divisible}.

(i) and (ii) imply (iii). Let $\beta$ and $\delta$ be actions of $G$ on $A$ as in the statement. Since $\M_{|G|^{\infty}}$ is a strongly self-absorbing algebra, there exists an isomorphism $\phi\colon A\to A\otimes \M_{|G|^{\infty}}$ that is approximately unitarily equivalent to the map $\iota\colon A\to A\otimes \M_{|G|^{\infty}}$ given by $\iota(a)=a\otimes 1_{\M_{|G|^{\infty}}}$ for $a$ in $A$. In particular, one has $\Cu(\phi)=\Cu(\iota)$. Hence, for every $a\in (A\otimes\K)_+$ we have
$$(\Cu(\phi)\circ \Cu(\beta_g))([a])=\Cu(\iota)[(\beta_g\otimes \id_{\K})(a)]=\left[((\beta_g\otimes \id_{\K})(a))\otimes 1_{\M_{|G|^\infty}}\right]$$
and
\begin{align*}
(\Cu(\delta_g\otimes \mu^G)\circ\Cu(\phi))([a])&=\Cu(\delta_g\otimes \mu^G)\left(\left[a\otimes 1_{\M_{|G|^\infty}}\right]\right)\\
&=\left[((\delta_g\otimes \id_{\K})(a))\otimes 1_{\M_{|G|^\infty}}\right].
\end{align*}
Since $\Cu(\beta_g)=\Cu(\delta_g)$ for all $g\in G$, it follows that
$$\Cu(\phi)\circ \Cu(\beta_g)=\Cu(\delta_g\otimes \mu_g)\circ\Cu(\phi)$$
for all $g$ in $G$. In other words, the $\CCu$-isomorphism $\Cu(\phi)\colon \Cu(A)\to \Cu(A\otimes \M_{|G|^\infty})$ is equivariant. Therefore, by the unital case of Theorem \ref{classif Rp on NCCW}, there exists an isomorphism $\varphi\colon A\to A \otimes \M_{|G|^\infty}$ such that $\varphi\circ\beta_g=(\delta\otimes\mu^G)_g\circ\varphi$ for all $g\in G$, showing that $\beta$ and $\delta\otimes\mu^G$ are conjugate.

(iii) implies (i). The existence of an action $\beta\colon G\to\Aut(A)$ with the Rokhlin property implies the existence of an isomorphism $A\to A\otimes \M_{|G|^\infty}$, simply by taking $\delta=\beta$.

The last claim follows immediately from (iii).
\end{proof}

\end{document}